\documentclass[reqno]{amsart}

\usepackage[utf8]{inputenc}
\usepackage[english]{babel}
\usepackage{amssymb}
\usepackage{hyperref}
\usepackage{enumitem}
\usepackage{etoolbox}
\usepackage{braket} 
\usepackage{faktor}
\usepackage{mathtools}
\usepackage{stackrel}
\usepackage{verbatim}
\usepackage{tikz-cd}
\usepackage[english]{cleveref}
\usepackage{todonotes}

\theoremstyle{definition}
\newtheorem{theorem}{Theorem}[section]
\newtheorem{thmx}{Theorem}

\newtheorem*{theo*}{Theorem}
\newtheorem{definition}[theorem]{Definition}
\newtheorem{remark}[theorem]{Remark}

\newtheorem{lemma}[theorem]{Lemma}
\newtheorem{proposition}[theorem]{Proposition}
\newtheorem{corollary}[theorem]{Corollary}
\newtheorem{conjecture}[theorem]{Conjecture}

\newcommand{\Z}{\mathbb{Z}}
\newcommand{\N}{\mathbb{N}}
\newcommand{\Q}{\mathbb{Q}}
\newcommand{\C}{\mathbb{C}}
\newcommand{\R}{\mathbb{R}}
\newcommand{\A}{\mathcal{A}} 
\newcommand{\B}{\mathbf{B}} 
\newcommand{\D}{\mathcal{D}} 
\newcommand{\E}{\mathcal{E}} 
\newcommand{\PP}{\mathcal{P}} 
\newcommand{\LLL}{\mathcal{L}} 
\newcommand{\hLL}{\hat{\mathcal{L}}_{>0}} 
\newcommand{\tH}{\tilde{H}} 
\newcommand{\UA}{\cup \A} 
\newcommand{\Fil}{\mathrm{L}} 
\newcommand{\MM}{\mathrm{M}} 
\newcommand{\DD}{\mathrm{D}} 
\newcommand{\SG}{\mathfrak{S}} 

\DeclareMathOperator{\Vol}{Vol} 
\DeclareMathOperator{\Ann}{Ann} 
\DeclareMathOperator{\cd}{cd} 
\DeclareMathOperator{\Top}{Top}
\DeclareMathOperator{\sgn}{sgn}
\DeclareMathOperator{\colim}{colim}
\DeclareMathOperator{\hcolim}{hcolim}
\DeclareMathOperator{\gr}{gr}
\DeclareMathOperator{\rk}{rk}
\DeclareMathOperator{\im}{Im} 
\DeclareMathOperator{\TT}{T} 
\DeclareMathOperator{\dd}{d} 

\makeatletter
\newcommand{\mylabel}[2]{#2\def\@currentlabel{#2}\label{#1}}
\makeatother

\makeatletter
\newcommand*{\bigcdot}{%
  {\mathbin{\mathpalette\bigcdot@{}}}%
}
\newcommand*{\bigcdot@scalefactor}{.75}
\newcommand*{\bigcdot@widthfactor}{1.4}
\newcommand*{\bigcdot@}[2]{%
  \sbox0{$#1\vcenter{}$}
  \sbox2{$#1\cdot\m@th$}%
  \hbox to \bigcdot@widthfactor\wd2{%
    \hfil
    \raise\ht0\hbox{%
      \scalebox{\bigcdot@scalefactor}{%
        \lower\ht0\hbox{$#1\bullet\m@th$}%
      }%
    }%
    \hfil
  }%
}
\makeatother

\robustify{\bigcdot}

\begin{document}

\title[]{On the cohomology of arrangements of subtori}

\author[L.\ Moci]{Luca Moci}
\thanks{The authors are supported by PRIN 2017YRA3LK}
\address{Luca Moci \newline \textup{Università di Bologna, Dipartimento di Matematica}\\ Piazza di Porta San Donato 5 - 40126 Bologna\\ Italy.}
\email{luca.moci2@unibo.it}

\author[R. Pagaria]{Roberto Pagaria}
\address{Roberto Pagaria \newline \textup{Università di Bologna, Dipartimento di Matematica}\\ Piazza di Porta San Donato 5 - 40126 Bologna\\ Italy.}
\email{roberto.pagaria@unibo.it}

\begin{abstract}
Given an arrangement of subtori of arbitrary codimension in a torus, we compute the cohomology groups of the complement.
Then, using the Leray spectral sequence, we describe the multiplicative structure on the graded cohomology.
We also provide a differential model for the cohomology ring by considering a toric wonderful model and its Morgan algebra. Finally we focus on the divisorial case, proving a new presentation for the cohomology of toric arrangements.
\end{abstract}

\maketitle

\section*{Introduction}
The cohomology ring of the complement of an arrangement of affine hyperplanes in a complex vector space admits a renowed combinatorial presentation in term of the poset of intersections of the arrangement, due to Orlik and Solomon \cite{OS80}. For a \emph{toric arrangement}, i.e.\ a collection of 1-codimensional subtori in a complex algebraic torus, a similar presentation was recently provided by \cite{CDDMP19}.

A different way of genealizing arrangements of hyperplanes is considering a family of affine subspaces, not necessarily of codimension 1. The complement of such a \emph{subspace arrangement} was studied by several authors (see \cite{GMP88,DCP95,Yu02,Yu99,DGMP00,dLS01}; see also  \cite{Bj} and the bibliography therein). In particular, Goresky and MacPherson provided the following description of the cohomology groups:

\begin{thmx}[{\cite[III.1.5, Theorem A]{GMP88}}] \label{Thm:GMP}
Let $\A$ be a subspace arrangement in $\R^d$, and let $M_\A = \R^d \setminus \UA$ be its complement.
The cohomology of the complement is given by
\[\tH^k(M_\A;\Z) \cong \bigoplus_{W \in \LLL_{>\hat{0}}} \tH_{\cd_{\R} W -k-2} \left( \Delta(\hat{0},W);\Z \right),\]
where $\LLL_{>\hat{0}}$ is the poset of flats $\LLL$ without the minimum $\hat{0}=\R^d$, and  $\Delta(\hat{0},W)$ is the order complex of the corresponding interval.
\end{thmx}

Later, De Concini and Procesi constructed in \cite{DCP95} a wonderful model for subspace arrangements, i.e.\ a smooth projective variety $Y$ that contains $M_\A$ as an open subset whose complement is a simple normal crossing divisor.
They also applied a result of Morgan \cite{Morgan78} to present the rational cohomology ring of $M_\A$ as the cohomology of a differential graded algebra explicitly constructed from the combinatorial data only.

In 1996, Yuzvinsky simplified the differential graded algebra (see \cite{Yu02}) by using the order complex of the poset of intersection. He also showed a connection between the results of \cite{GMP88} and of \cite{DCP95}.
A further simplification was obtained in \cite{Yu99} by replacing the order complex with the atomic complex. 

Yuzvinsky also conjectured an integral version of this presentation. This conjecture was proven in \cite{DGMP00} and  \cite{dLS01}:
Deligne, Goresky, and MacPherson proved their result using diagram of spaces, de Longueville and Schultz by showing that the isomorphism of \Cref{Thm:GMP} is canonical.

In this paper we consider arrangements of subtori of arbitrary codimension in a complex algebraic torus. Given the complement of such an arrangement, we determine its cohomology groups (see  \Cref{thm:main_additive}):

\begin{thmx}
Let $\A$ be an arrangements of subtori of a torus $T$ and $\LLL$ be it poset of layers. Then the cohomology groups of the complement $M_\A$ are
\[ H^k(M_\A;\Z) \cong \bigoplus_{W \in \LLL} \:\bigoplus_{p+q=k} H^p(W;\Z) \otimes_\Z \tH_{2 \cd W - 2-q}(\Delta(T,W)).\]
\end{thmx}

Our proof is based on a suitable embedding of a $d$-dimensional complex algebraic torus $T$ in the Alexandroff compactification of $\C^d$, that is, the sphere $\mathbb S^{2d}$. The embedding is chosen so that the complement of $T$ in $\mathbb S^{2d}$ does not intersect the toric subspaces; hence the arrangement decomposes in a wedge of two simpler ones, given respectively by the compactifications of the coordinate hyperplanes and of the subtori in the original arrangement (\Cref{Lemma:embedding}). Then we apply some results on homotopy colimits \cite{WZZ99}, following a strategy outlined by Deshpande in \cite{Deshpande18}.
In that paper the same result is announced, but the proof therein does not seem to be correct, as several steps fail if the compactification is not chosen carefully.

Moreover we describe the multiplicative structure on the graded of the cohomology, by using the Leray spectral sequence for the inclusion map $j \colon M_\A \to T$.
We show, by using results of the previous section, that the second page of the spectral sequence is a finitely generated $\Z$-module isomorphic  to the cohomology as a module.
It follows that the spectral sequence degenerates at the second page and this gives the isomorphism $E_2^{p,q} \cong \gr_{p+q}^{\Fil} H^{p+q}(M_\A;\Z)$ (\Cref{thm:spectral}).

Furthermore, we provide a model for the cohomology of  $M_\A $: we use the wonderful model for toric arrangements introduced by De Concini and Gaiffi (see \cite{DCG1,DCG2}) to construct a differential graded algebra whose cohomology is isomorphic to the rational cohomology ring of the complement (\Cref{dga-iso}).
Since our method are based on the aforementioned Morgan algebra, this d.g.a.\ codifies also the rational homotopy type of the complement.

Finally we focus on the case of  an arrangement of subtori of codimension $1$. Given such a \emph{toric arrangement}, and chosen its maximal \emph{building set}, we find a subalgebra of the Morgan model isomorphic to the cohomology ring.
This subalgebra, explicitly presented by generators and relations in \Cref{thm:main_div}, yields an analogue of the Orlik-Solomon algebra for toric arrangements.
This new presentation depends on the oriented arithmetic matroid only (see \cite{OAM20}) and, compared to the previous result of \cite{CDDMP19}, exhibits more clearly the dependence from the orientation. Furthermore, it seems more suitable to be generalized to arrangement of subtori of arbitrary codimensions. We also conjecture that a similar presentation holds for cohomology with integer coefficients (\Cref{conj-int}).
\newpage
\tableofcontents

\section{Positive systems and embedding of subtori}
A $d$-dimensional \emph{complex torus} $T$ is an algebraic group isomorphic to $(\C^*)^d$. A \emph{character} is a morphism of algebraic groups $T\to\C^*$. The group $\Lambda$ of all characters is a lattice of rank $d$, i.e., it is isomorphic to $\Z^d$. A \emph{subtorus} of $T$ is a translate of a subgroup isomorphic to $(\C^*)^k, \: 0\leq k< d$. 

\begin{definition}
An \emph{arrangement of subtori} is a finite collection $\A=\{S_1, \dots, S_n \}$ of subtori of $T$ such that $S_i \not \subseteq S_j$ for all $i\neq j$.
\end{definition}

We denote by $M_\A$ the complement of this arrangement, that is, 
\[\mathbb C^n\setminus (S_1\cup\dots\cup S_n).\]

The set of characters that are constant on a subtorus $S_i$ is a subgroup of $\Lambda$, that we denote by $\Lambda_i$.
Let $\mathbf B$ be a basis (over $\Z$) of $\Lambda$ and, for every $i=1,\dots,n$, let $\mathbf B_i$ be a basis (over $\Z$) of $\Lambda_i$.

\begin{definition}
We say that $(\mathbf B, \mathbf B_1,\dots, \mathbf B_n)$ is a \emph{positive system} if the coordinates of all the elements of every $\mathbf B_i$ in the basis $\mathbf B$ are positive.
\end{definition}
The above definition is inspired by similar (and indeed stronger) notions introduced by De Concini and Gaiffi in \cite{DCG1,DCG2}.

\begin{lemma}
Every arrangement of subtori admits a positive system.
\end{lemma}

\begin{proof}
For each $S_i \in \A$ choose a basis $\mathbf{B}_i$ of $\Lambda_i$ and a basis $\mathbf{B}$ of $\Lambda$.
Consider the matrix $A$ that represent the elements $b_{i,j}\in \mathbf{B}_i$, for $i=1, \dots, n$ and $j=1,\dots , |\mathbf{B}_i|$ in the basis $\mathbf{B}$.
By changing $b_{i,j}$ with $-b_{i,j}$ we suppose that the last non-zero entry of the $j-$th column of $A$ is a positive integer: we call this entry the pivot of the column.
We perform a sequence of elementary row operation in order to make $A$ a non-negative matrix.
The columns with pivot in the first row are already non-negative.
We proceed by induction, suppose that all columns with pivot in the first $k-1$ rows are non-negative.
By adding a suitable multiple of the $k$th row to the previous rows, we can make all the columns with pivot in the $k$th row non-negative. This operation does not change the columns with pivot in the first $k-1$ rows.
By repeating the procedure for every $k=2, \dots, d$, we obtain a non-negative matrix.
The elementary row operations correspond to a change of basis from $\mathbf{B}$ to a new basis $\mathbf{B}'$ that form a positive system $(\mathbf{B}',\mathbf B_1,\dots, \mathbf B_n)$.
\end{proof}

We denote by $\mathbb{S}^d$ the $d$-dimensional real sphere, and by $\mathcal B_d$ the Boolean arrangement, i.e. the set of the coordinate hyperplanes in $\mathbb C^d$.

Given a topological space $X$, its \emph{Alexandroff compactification} $\widehat{X}$ is defined as the topological space on the set $X\cup\{\infty\}$ whose basis of open sets is given by the open sets of $X$ and the sets $(X\setminus C)\cup \{\infty\}$, where $C$ ranges over all the closed and compact sets of $X$.
For instance, the Alexandroff compactification of $\C^d$ is isomorphic to the sphere $\mathbb S^{2d}$.

\begin{proposition}\label{Lemma:embedding}
Let $\A$ be a arrangements of subtori in a $d$-dimensional torus $T$.
Then there exists an embedding $M_\A \hookrightarrow \mathbb{S}^{2d}$ such that 
\[\mathbb{S}^{2d} \setminus M_\A = \widehat{\UA} \:\vee\: \widehat{\cup \mathcal{B}_{2d}}.\]
\end{proposition}

\begin{proof}
We choose a positive system $(\mathbf B, \mathbf B_1,\dots, \mathbf B_n)$.
The basis $\B$ defines an isomorphism $T \cong (\C^*)^d$, and consider the composition 
\[M_{\A} \subset T \cong (\C^*)^d \subset \C^d \subset \widehat{\C^d}\cong \mathbb{S}^{2d}.\]
Notice that $\C^d \setminus M_{\A}$ is the disjoint union of $\UA$ and $\cup \mathcal{B}_{2d}$ because the system
\[ \begin{cases}
z_1^{n_1} z_2^{n_2} \cdots z_d^{n_d}=c\\
z_j = 0
\end{cases}\]
for $n_i \in \N$, $c \in \C^*$, and $j\leq d$, has no solutions.
The condition of positive system ensure that each subtorus $S_i \in \A$ is contained in a hypertorus of the form \[\{\underline{z} \in (\C^*)^d \mid z_1^{n_1} z_2^{n_2} \dots z_d^{n_d}=c\}\]
for some $c \in \C^*$ and some $n_i \in \N$, not all equal to zero.
Now, $\mathbb{S}^{2d} \setminus M_\A$ is the Alexandroff compactification of $\C^d \setminus M_{\A}$, hence 
\[ \mathbb{S}^{2d} \setminus M_\A \cong \widehat{\C^d \setminus M_{\A}} \cong \widehat{\UA \sqcup \cup \mathcal{B}_{2d}} \cong \widehat{\UA} \vee \widehat{\cup \mathcal{B}_{2d}}. \qedhere\]
\end{proof}

\section{Cohomology groups of arrangements of subtori}

Let $\PP$ be a poset. We recall that the \emph{order complex} $\Delta(\PP)$  is the simplicial complex 
whose simplexes are the the totally ordered subsets of $\PP$.
For any $W,L \in \PP$ with $W>L$ we denote $\Delta(L,W)$ the order complex of the sub-poset 
\[\set{X \in \PP \mid W>X>L}.\]

\begin{definition}
Given two pointed CW-complexes $(X,x)$ and $(Y,y)$, we define:
\begin{itemize}
\item the \emph{wedge sum} $X\vee Y$ as $X\sqcup Y/x\sim y$;
\item the \emph{smash product} $X\wedge Y$ as the topological quotient $X\times Y/ X\vee Y$;
\item the \emph{join} $X*Y$ as $X\wedge Y \wedge \mathbb{S}^1$.
\end{itemize}
\end{definition}

Let $\hLL$ be the poset obtained from the poset of layers $\LLL$ by removing the minimum $\hat{0}=T$ and adding a maximum $\hat{1}$. We think of $\hLL$ as a category, having one morphism $p \to q$ for every $p, q\in \hLL$ such that $p\leq q$. 

Given a poset $\PP$, a $\PP$-\emph{diagram} is a functor from the category $\PP$ to the category $\Top_*$ of pointed topological spaces.

\begin{definition}
We define two $\hLL$-diagrams $\D$ and $\E$ as follows.

For every object $W \in \hLL$,
\[ \D(W)=\E(W)= \widehat{W} \textnormal{ and } \D(\hat{1})=\E(\hat{1})= \{\infty\};\]
for every map $W>L$, $\D$ is defined by the natural inclusions:
\[ \D(W>L)= \widehat{W} \hookrightarrow \widehat{L} \textnormal{ and } \D(\hat{1}>W) = \{\infty \} \hookrightarrow \widehat{W} \]
while $\E$ by the constant maps at the point $\infty$: 
\[ \E(W>L)= \widehat{W} \to \widehat{L} \textnormal{ and } \E(\hat{1}>W) = \{\infty \} \to \widehat{W}. \]
\end{definition}
The \emph{colimit} of a $\mathcal{P}$-diagram $\mathcal{F}$ is the union of all topological spaces $\mathcal{F}(p)$ for all $p \in \mathcal{P}$ with the identification given by the maps between them (i.e.\ $x = \mathcal{F}(p \to q)(x)$ for all maps $p \to q$ in $\mathcal{P}$ and all $x \in \mathcal{F}(p)$).

The \textit{homotopy colimit} of $\mathcal{F}$ can be constructed by replacing all the maps $\mathcal{F}(p \to q)$ with homotopy equivalent cofibrations and then taking the colimit of the resulting diagram.

We recall the following results of Welker, Ziegler and \v{Z}ivaljevi\'{c}:

\begin{lemma}[{Projection Lemma \cite[Lemma 4.5]{WZZ99}}] \label{ProjLemma}
Let $\D$ be a $\PP$-diagram such that all maps are inclusions and closed cofibrations.
Then the natural map $\hcolim \D \to \colim \D$ from the homotopy colimit to the colimit of $\D$  is a homotopy equivalence.
\end{lemma}

\begin{lemma}[{Homotopy Lemma \cite[Lemma 4.6]{WZZ99}}] \label{HomotLemma}
Let $\D$ and $\E$ be $\PP$-diagrams and $h \colon \D \to \E$ be a morphisms of $P$-diagrams (i.e.\ a natural transformation between the two functors).
Suppose that for all $W\in \PP$ the map $h_W \colon \D(W) \to \E(W)$ is a homotopy equivalence, then the induced map $\hcolim \D \to \hcolim \E$ is a homotopy equivalence. 
\end{lemma}

\begin{lemma}[{Wedge Lemma \cite[Lemma 4.9]{WZZ99}}] \label{WedgeLemma}
Let $\PP$ be a poset with a maximal element and let $\E$ be a $\PP$-diagram.
Suppose that all maps in $\E$ are constant morphisms of pointed spaces, then
\[ \hcolim \E \simeq \bigvee_{p \in \PP}\left( \Delta(\PP_{<p}) * \E(p)\right).\]
\end{lemma}

We now prove the following result on compactifications of subtori:

\begin{lemma} \label{Lemma:homotopy_equivalence}
Let $\A$ be an arrangement of subtori.
For each $W \in \LLL$ there exists a homotopy equivalence $h_W \colon \widehat{W} \to \widehat{W}$ such that, for all $L>W$, the following diagram commutes.
\begin{center}
\begin{tikzcd}
\widehat{L} \ar[r, hook] \ar[d] & \widehat{W} \ar[d, "h_W"] \\
\{\infty\} \ar[r] & \widehat{W}
\end{tikzcd}
\end{center}
\end{lemma}

\begin{proof}
Consider a positive system $\mathbf{B}, \mathbf{B}_i$ for the restricted arrangement in $W$
\[\A^W =\{S \cap W \mid S \in \A, S \cap W \neq \emptyset, W \},\]
i.e.\ a basis $\mathbf{B}$ of $\Lambda_W$ and a basis $\mathbf{B}_i$ of $\Lambda_{S_i}$ for each atom $S_i \in \A_W$.
The basis $\mathbf{B}$ gives an isomorphism between $W$ and $(\C^*)^{\dim W}$.
Let $\epsilon \in \R^+$ be the minimum of the distance between $0 \in \C^{\dim W}$ and $S_i$ for all atoms $S_i \in \A_W$.

Each layer $L\in \LLL_{>W}$ of the restricted arrangement $\A_W$ is contained in a hypertorus 
\[\{\underline{z} \in \C^{\dim W} \mid z_1^{n_1}\dots z_{\dim W}^{n_{\dim W}} = c \}\]
for some $n_i \in \N$ and some $c \in \C^*$.
Hence $\epsilon$ is positive and each layer $L$ is disjoint to the ball $D_\epsilon \subset \C^{\dim W}\subset \mathbb{S}^{2\dim W}$ of center $0$ and radius $\epsilon$.

Choose a continuous, increasing and surjective function $f\colon [0,\epsilon) \to [0, \infty)$ and define $\tilde{h}_W \colon \mathbb{S}^{2\dim W} \to \mathbb{S}^{2\dim W}$ by 
\[ x \mapsto 
\begin{cases}
f(|x|)x & \textnormal{if } x \in D_\epsilon, \\
\infty & \textnormal{otherwise},
\end{cases}\]
where $|x|$ is the distance of $x$ from $0$ and $\infty$ is the unique point in $\mathbb{S}^{2\dim W} \setminus \C^{2\dim W}$.
It easy to see that $\tilde{h}_W$ induces a homotopy equivalence $h_W \colon \widehat{W} \to \widehat{W}$.
The commutativity of the diagram above follows from $L \cap D_\epsilon = \emptyset$.
\end{proof}

The previous results now allow us to describe the Alexandroff compactification of the union of the subtori of the arrangement:

\begin{lemma} \label{Lemma:decomposition_UA}
There exists a homotopy equivalence
\[ \widehat{\UA} \simeq \bigvee_{w \in \hLL} \left( \widehat{W}* \Delta(T,W) \right)\]
\end{lemma}
\begin{proof}
Consider the maps $h_W$ given by \Cref{Lemma:homotopy_equivalence} and let $h_{\hat{1}}\colon \{\infty\} \to \{\infty\}$ be the only map.
This data define a morphism $h \colon \D \to \E$ of $\hLL$-diagrams.
We have
\[ \widehat{\UA} \simeq \colim \D \simeq  \hcolim \D \simeq \hcolim \E \simeq \bigvee_{w \in \hLL} \left( \widehat{W}* \Delta(T,W) \right), \]
where the first isomorphism follow by the definition of $\colim$, the others by the projection \Cref{ProjLemma}, the homotopy \Cref{HomotLemma} applied to $h$, and the wedge \Cref{WedgeLemma} respectively.
\end{proof}

\begin{theorem}\label{thm:main_additive}
Let $\A$ be an arrangements of subtori of a torus $T$ and $\LLL$ be it poset of layers. Then the cohomology groups of the complement $M_\A$ are
\[ H^k(M_\A;\Z) \cong \bigoplus_{W \in \LLL} \:\bigoplus_{p+q=k} H^p(W;\Z) \otimes_\Z \tH_{2 \cd W - 2-q}(\Delta(T,W)))\]
\end{theorem}

\begin{proof}
Consider the embedding $M_\A \subset S^{2d}$ of such that $S^{2d} \setminus M_\A = \widehat{\UA} \vee \widehat{B_d}$, provided by \Cref{Lemma:embedding}.
We use the Alexander duality  (see for instance \cite[Theorem 3.44]{Hatcher}) to obtain
\[\tH^k(M_\A)\cong \tH_{2d-k-1}(\widehat{\UA} \vee \widehat{B_d}) \cong  \tH_{2d-k-1}(\widehat{\UA}) \oplus \tH_{2d-k-1}(\widehat{B_d}).\]
Again Alexander duality for the embedding $\widehat{B_d} \subset S^{2d}$ gives the isomorphism $\tH_{2d-k-1}(\widehat{B_d})\cong \tH_{k}(T)$.
Now, \Cref{Lemma:decomposition_UA} implies 
\begin{align*}
\tH_{2d-k-1}(\widehat{\UA}) &\cong \tH_{2d-k-1} \Big( \bigvee_{w \in \hLL} \left( \widehat{W}* \Delta(T,W) \right) \Big) \\
& \cong \bigoplus_{w \in \hLL} \tH_{2d-k-1}\left( \widehat{W}* \Delta(T,W) \right) \\
& \cong \bigoplus_{w \in \hLL} \tH_{2d-k-2}\left( \widehat{W} \wedge \Delta(T,W) \right) \\
& \cong \bigoplus_{w \in \hLL} \bigoplus_{p+q=k} \tH_{2\dim W-p} ( \widehat{W}) \otimes_\Z \tH_{2\cd W-q-2}(\Delta(T,W))),
\end{align*}
where the last isomorphism is the Kunneth isomorphism for reduced cohomology applied to the smash product.
We conclude the proof by the Poincarè duality on $W$ between Borel-Moore homology and cohomology (see \cite[Theorem 9.2]{Bredon}):
\[\tH_{2\dim W-p} ( \widehat{W}) = H^{BM}_{2\dim W-p} (W) \cong H^p(W). \qedhere\]
\end{proof}

\section{Graded of the cohomology ring}
In this section we study the Leray spectral sequence for  the inclusion map $j \colon M_\A \to T$ to give a description of the graded cohomology ring $\gr^\Fil_\bigcdot H^\bigcdot (M_A;\Z)$.
We refer to \cite[Sect.~6]{Bredon} as a general reference on this spectral sequence.

Let $R^q j_*$ be the higher direct image functor of $j$.
In our case the Leray spectral sequence
\[E_r^{p,q} \Rightarrow H^{p+q}(M_A;\Z) \]
converges to $H^{p+q}(M_A;\Z)$; the second page of this spectral sequence is 
\[E_2^{p,q}=H^p(R^q j_* \Z_{M_\A}),\]
where $\Z_{M_\A}$ is the constant sheaf and the Leray filtration $\Fil^\bigcdot$ on $H^{k}(M_A;\Z)$ is defined by
\[\Fil^q = \im \left( H^k (T;\tau_{k-q} \mathbb{R}j_*\Z)\to H^k (T; \mathbb{R}j_*\Z) \cong H^k(M_\A;\Z) \right)\]

For each $W \in \LLL$, let $\epsilon^q_W$ be the sheaf on $T$ defined by 
\[\epsilon^q_W = (i_W)_* \Z_W \otimes_Z H_{2 \cd W -q-2}(\Delta(T,W))\]
where $i_W$ is the closed embedding of $W$ in $T$.
We set 
\[\epsilon^q= \bigoplus_{W \in \LLL} \epsilon^q_W.\]
The following lemma generalizes \cite[Lemma 3.1]{Bibby16}.
\begin{lemma}\label{lemma:higher_direct_image}
Let $\A$ be an arrangement of subtori. Then
\[E_2^{p,q}=H^p(R^q j_* \Z_{M_\A}) \cong \bigoplus_{W \in \LLL} H^p(W;\Z) \otimes \tH_{2\cd W-q-2}(\Delta(T,W);\Z). \]
\end{lemma}
\begin{proof}
First we prove that $\epsilon^q \cong R^q j_* \Z_{M_\A}$: for each point $x \in T$ there exists an open set $U_x$ isomorphic to an open subset $V_x$ of the tangent space $\TT_x T$ (containing the origin).
We also take a neighborhood basis $\mathcal{U}$ given for every $x\in T$ by the inverse image of all open balls in $V_x$ centered in $0 \in \TT_x T$.
Notice that the arrangement of subtori $\A$ defines a central arrangement of subspaces 
\[\A[x]=\set{\TT_x S \mid x \in S \in \A}\]
in $\TT_x T$ for all $x \in T$.

We define a morphism of sheaves $f\colon \epsilon^q \to R^q j_* \Z_{M_\A}$ on the neighborhood basis $\mathcal{U}$ as follow.
For all $U \in \mathcal{U}$ centered in $x$, let $f(U)\colon \epsilon(U) \to R^qj_*\Z(U)$ be the composition
\[\epsilon(U) = \bigoplus_{W \ni x} \Z \otimes_\Z \tH_{2 \cd W -q-2}(\Delta(T,W)) \cong H^q(M_{\A[x]}) \cong H^q(U \cap M_\A)= R^q j_*\Z (U), \]
where the first isomorphism is given by \Cref{Thm:GMP} and the second one is given by the composition 
\[M_{\A[x]} \simeq V_x \setminus \cup \A[x] \cong U \setminus \cup \A.\]
Since $f(U)$ is an isomorphism for all $U \in \mathcal{U}$ then $f$ is an isomorphism of sheaves.
Now, the isomorphism 
\[H^p(\epsilon^q) \cong \bigoplus_{W \in \LLL} H^p(W;\Z) \otimes \tH_{2\cd W-q-2}(\Delta(T,W);\Z)\]
completes the proof.
\end{proof}

The minimum of the poset $\LLL$ (and of $\LLL_{\leq W}$ for all $W\in \LLL$) is $\hat{0}=T$.
Let $E_W^{p,q}\subset E_2^{p,q}$ be the $\Z-$module $H^p(W;\Z) \otimes \tH_{2\cd W-q-2}(\Delta(T,W);\Z)$: this module depends only on the cohomology of $W$ and on the poset $\LLL_{\leq W}$.

The multiplication in $E_2$ is induced by the maps
\[\eta_{W,W'}^L \colon E_W^{p,q} \otimes E_{W'}^{p',q'} \to E^{p+p',q+q'}_L\]
where $\eta_{W,W'}^L=0$ if $\cd L \neq \cd W + \cd W'$ or if $L$ is not a connected component of $W \cap W'$, otherwise 
\[\eta_{W,W'}^L(a \otimes b \otimes a' \otimes b') = (-1)^{p'q} (a\smile a') \otimes (b \bullet b')\]
where $\bullet \colon \tH_k(\Delta(\hat{0},W)) \otimes \tH_{k'}(\Delta(\hat{0},W')) \to \tH_{k+k'+2}(\Delta(\hat{0},L))$ is the map of \cite[Theorem 6.6 (ii)]{Yu02}.
Under the isomorphism of \Cref{Thm:GMP} the map $\bullet$ corresponds to the cup product in the cohomology of the subspace arrangement $\A[x]$ (for any $x \in L$).

\begin{theorem}\label{thm:spectral}
The Leray spectral sequence $E_r^{p,q}$ for the inclusion $M_\A \hookrightarrow T$ degenerates at the second page, i.e. 
\[E_2^{p,q} \cong \gr^{\Fil}_{p+2q} H^{p+q} (M_\A ;\Z).\]
\end{theorem}
\begin{proof}
We know that $E_{\infty}^{p,q}$ is a subquotient of $E_2^{p,q}$ and that the last page is $E_{\infty}^{p,q}\cong \gr^{\Fil}_{p+2q} H^{p+q} (M_\A ;\Z)$.
By \cref{thm:main_additive,lemma:higher_direct_image}, $E_{\infty}^{p,q}$ and $E_2^{p,q}$ are isomorphic and finitely generated; hence $E_2^{p,q}=E_{\infty}^{p,q}$.
\end{proof}

\section{A model for the complement} 

As in the previous sections, we denote by $\A$  an arrangement of subtori in $T$ and by $\Lambda$ the character group of $T$. Let $\Lambda^*$ be the dual lattice of $\Lambda$. We refer to \cite{CLSbook} for a general introduction to fans and toric varieties.
Let $\Delta$ be a smooth and complete fan in $\Lambda^*$.
Every ray of $\Delta$ is generated by a (uniquely determined) primitive vector in $\Lambda^*$: we denote by $\mathcal{R}_\Delta \subset \Lambda^*$ the set of primitive vectors corresponding to the rays of $\Delta$. Let $\mathcal{P}(\mathcal{R}_\Delta)$ be its set of parts; we denote by $\mathcal{C}_\Delta \subseteq \mathcal{P}(\mathcal{R}_\Delta)$ the collection of the sets of primitive vectors that span a cone in $\Delta$. Thus from now on we identify a cone in $\Delta$ with the set of its extremal primitive vectors.

\begin{definition}
A fan $\Delta$ in $\Lambda^*$ describes a \textit{good toric variety} $X_\Delta$ (with respect to $\A$) if $\Delta$  is complete and smooth and each maximal cone $C \in \mathcal{C}_\Delta$ can be completed to a positive system $(\mathbf{C}, \mathbf{C}_1, \dots, \mathbf{C}_n)$ (where $\mathbf{C}$ is the dual basis of $C$).
\end{definition}
The second condition in the above definition can be reformulated as follow: for each $W\in \A$, there exists a basis $\beta_1, \dots, \beta_{\cd W}$ of $\Lambda_W$ such that for each maximal cone $C \in \mathcal{C}_\Delta$ and each $i=1, \dots, \cd W$ the natural pairing $\langle \beta_i, c \rangle$ is nonnegative (or nonpositive) for all $c \in C$.
In this case, we say that the basis $\beta_1, \dots, \beta_{\cd W}$ of $\Lambda_W$ has the \textit{equal sign property} with respect to $\Delta$ (see \cite[Definition 3.2]{DCG1}).

Let $\mathcal{G}\subseteq \LLL_{>0}$ be a \textit{well connected building set} in the sense of \cite[Definition 4.1]{DCG2} and $\Delta$ a good toric variety.
These data define a smooth projective variety $Y(\Delta, \mathcal{G})$ obtained from $X_\Delta$ by subsequently blowing up (the strict transform of) $W$ for all $W\in \mathcal{G}$ in any total order refining the partial order given by inclusion (so that smaller layers are blown up first).
The variety $Y(\Delta, \mathcal{G})$ is the \emph{wonderful model} for $M_\A$ described in \cite{DCG2}, i.e. a smooth projective variety containing $M_\A$ and such that the complement $Y(\Delta, \mathcal{G}) \setminus M_\A$ is a simple normal crossing divisor.

We want to describe the Morgan algebra (cf \cite{Morgan78}) for the pair $(Y(\Delta, \mathcal{G}),M_\A)$. For the convenience ot the reader, we will briefly recall here the definition of this algebra.
Consider a smooth complete algebraic variety $Y$ and a simple normal crossing divisor $D=\bigcup_{i\in I} D_i$ with complement $M$.
The \textit{Morgan algebra} (see \cite{Morgan78}) is the vector space 
\[\bigoplus_{A\subseteq I} H^\bigcdot\left(\bigcap_{i \in A} D_i\right)\]
with $H^p(\bigcap_{i \in A} D_i)$ of bi-degree $(p,\lvert A \rvert )$.
The multiplication is given by 
\[H^p(\bigcap_{i \in A} D_i)\otimes H^{p'}(\bigcap_{i \in A'} D_i) \to H^{p+p'}(\bigcap_{i \in A\sqcup A'} D_i),\]
i.e. the composition of the restriction maps and the cap product.
The differential is induced by the Gysin morphisms $H^p(\bigcap_{i \in A} D_i) \to H^{p+2}(\bigcap_{i \in B} D_i)$ for every $B=A \setminus \{a\}$.

Let $E$ be the exterior algebra over $\Q$ on generators $s_W,t_W, b_j,c_j$ for $W \in \mathcal{G}$ and $j \in \mathcal{R}_\Delta$, where the bi-degree of $s_W$ and $b_j$ is $(0,1)$ and the bi-degree of $t_W$ and $c_j$ is $(2,0)$. 

In order to understand what relations should put on $E$, we start by recalling the definition of some polynomials $P_L^W(t)$, which were introduced in Section 8 of \cite{DCG2} as \textit{good lifting} of the Chern polynomials\footnote{The authors of \cite{DCG2} forgot to specify that, in order to define a good lifting, the basis of $\Lambda_W$ must have the equal sign property with respect to $\Delta$.}.

For each pair $W\leq L$ in $\LLL$, and for each t-uple $\beta_1, \dots, \beta_{\cd W} \in \Lambda_W$ with the equal sign property with respect to $\Delta$ such that $\beta_{\cd L- \cd W +1}, \dots, \beta_{\cd W}$  form a integral basis of $\Lambda_L$, define 

\[P^W_L(t)= \prod_{i=1}^{\cd L-\cd W} \left(t- \sum_{j\in \mathcal{R}_\Delta} \min(0, \langle \beta_i, j \rangle) c_j \right).
\]
The polynomial $P^W_L(t)$ depends on the choice of $\beta_1, \dots, \beta_{\cd L- \cd W}$.

\begin{definition}
A set $A\subset \mathcal{G}$ is \textit{nested} if the irreducible components of the normal crossing divisor of $Y(\Delta,\mathcal{G})$ that correspond to the elements of $A$  have non-empty intersection. When we want to emphasize the dependence on $\mathcal{G}$ we will say that $A$ is \emph{$\mathcal{G}$-nested}. 
\end{definition}
The property of being nested does not depend on the choice of $\Delta$, and can be expressed in a purely combinatorial way (see \cite[Definition 2.7]{DCG2}).

We recall the following result of De Concini and Gaiffi.
\begin{theorem}{\cite[Theorem 9.1]{DCG2}} \label{thm:DCG}
Let $A\subseteq \mathcal{G}$ and $B\subseteq \mathcal{R}_\Delta$, the intersection $Y_{A \sqcup B}$ is non-empty if and only if $A$ is $\mathcal{G}$-nested, $B$ is contained in $\bigcap_{W \in A} \Ann \Lambda_W$, and $B$ is a cone in $\Delta$.
In this case the cohomology ring $H^\bigcdot (Y_{A\sqcup B})$ is the exterior algebra on generators $\{t_W\}_{W \in \mathcal{G}}$ and $\{c_j\}_{j \in \mathcal{R}_\Delta}$ of degree two with relations:
\begin{enumerate}
    \item[\mylabel{item:prod_toric}{(T1)}] $\prod_{j \in C} c_j$ if $C \not \in \mathcal{C}_\Delta$,
    \item[\mylabel{item:linear_toric}{(T2)}] $\sum_{j \in \mathcal{R}_\Delta} \langle \beta, j \rangle c_j$ for every $\beta \in \Lambda$ (or equivalently for $\beta$ in a fixed basis of $\Lambda$),
    \item[\mylabel{item:c_wonderful}{(W1)}] $\prod_{j \in C} c_j$ if $C \cup B$ is not a cone in $\Delta$ or $C \not \subset \bigcap_{W \in A} \Ann \Lambda_W$,
\footnote{In \cite{DCG2} these relations are stated only for $\lvert C \rvert = 1$; however they hold, before performing blow-ups, for any set $C$, by the well-known theory of toric varieties.
Thus, by adding the relations with $\lvert C \rvert >1$ to the presentation given in \cite{DCG2}, we get a correct presentation.}
    \item[\mylabel{item:tc_wonderful}{(W2)}] $t_W c_j$ if $j \not \in \Ann \Lambda_W$,
    \item[\mylabel{item:P_wonderful}{(W3a)}] for all $W\in \mathcal{G}$ and all $C \subseteq \mathcal{G}_{<W}$, the relations 
    \[P_W^V \left( \sum_{L \in \mathcal{G}_{\geq W}} -t_L \right) \prod_{L \in C} t_L,\]
    where $V$ is the connected component of $\bigcap_{L \in A_{<W}\cup C} L$ containing $W$,
    \item[\mylabel{item:t_wonderful}{(W3b)}] $\prod_{W \in C} t_W$ if $C\cup A$ is not $\mathcal{G}$-nested or $B \not \subset \bigcap_{W \in C} \Ann \Lambda_W$.
\end{enumerate}
\end{theorem}

Although the polynomials $P_W^V(t)$ in \ref{item:P_wonderful} depend on the choice of a basis, the ideal generated by all the relation is independent from this choice, as shown in \cite[Proposition 6.3]{DCG2}.

\begin{remark}
Another possible choice of $P_W^V(t)$ consist of taking the polynomials:
\[t^{\cd L-\cd W} +  \prod_{i=1}^{\cd L-\cd W}  \sum_{j\in \mathcal{R}_\Delta} -\min(0, \langle \beta_i, j \rangle) c_j.\]
\end{remark}

Let $A$ be a nested set, $W$ be any element in $\mathcal{G}$, and $B \subseteq \mathcal{G}$ be such that each $L \in B$ is smaller than $W$ ($L \lneq W$ in $\LLL$).
We define the element $F(A,W,B)$ in $E$ by
\[ F(A,W,B) = P_W^V \left( \sum_{L \in \mathcal{G}_{\geq W}} -t_L \right) \prod_{L \in A} s_L \prod_{L \in B} t_L ,\]
where $V$ is the connected component of $\bigcap_{L \in A_{< W} \cup B} L$ containing $W$ (so $V\leq W$).

\begin{definition} \label{def:D}
Let $(\DD, \dd)$ be the differential graded algebra given by $E$ with relations:
\begin{enumerate}
    \item \label{item:first_def_M} $x_W y_j$ if $j \not \in \Ann \Lambda_W$, where $x_W= s_W$ or $t_W$ and $y_j=b_j$ or $c_j$,
    \item \label{item:second_def_M} $\prod_{W \in A} s_W \prod_{W \in B} t_W$ if $A\cup B$ is not a $\mathcal{G}$-nested set,
    \item \label{item:third_def_M} $\prod_{j \in A} b_j \prod_{j \in B} c_j$ if $A\cup B$ is not a cone in $\Delta$ (ie $A \cup B \in \mathcal{P}(\mathcal{R}_\Delta) \setminus \mathcal{C}_\Delta$),
    \item \label{item:four_def_M} $\sum_{j \in \mathcal{R}_\Delta} \langle \beta, j \rangle c_j$ for every $\chi \in \Lambda$ (or equivalently for $\chi$ in a fixed basis of $\Lambda$),
    \item \label{item:last_def_M} $F(A,W,B)$ for $A$ $\mathcal{G}$-nested set, $W \in \mathcal{G}$, and $B \subseteq \mathcal{G}$ be such that each $L \in B$ is smaller than $W$ (ie $B \subseteq \mathcal{G}_{< W}$),
\end{enumerate}
and differential $\dd$ defined on generators by $\dd(s_W)=t_W$, $\dd(b_j)=c_j$, and by $\dd(t_W)=\dd(c_j)=0$.
\end{definition}

\begin{lemma}
The ideal generated by \eqref{item:first_def_M}-\eqref{item:last_def_M} is stable with respect to $\dd$, so $(\DD,\dd)$ is a differential graded algebra.
\end{lemma}

\begin{proof}
It is obvious that the ideal generated by \eqref{item:first_def_M}-\eqref{item:four_def_M} is $\dd$-stable.
The relation 
\[ \dd(F(A,W,B)) = \sum_{L \in A_{<W}} \pm F(A\setminus \{L\}, W, B \cup \{L\}) +  \sum_{L \in A_{\not < W}} \pm t_L F(A\setminus \{L\}, W, B) \]
show that the ideal generated by \eqref{item:last_def_M} is $\dd$-stable.
\end{proof}

\medskip
Let $\MM$ be the Morgan algebra  associated to the pair $(Y(\Delta,\mathcal{G}), M_\A)$. The complement $Y(\Delta,\mathcal{G}) \setminus M_\A$ is a simple normal crossing divisor $\bigcup_{W \in \mathcal{G}} D_W \cup \bigcup_{j \in \mathcal{R}_{\Delta}} D_j $, whose irreducible component are indexed by $\mathcal{G} \sqcup \mathcal{R}_\Delta$.

For each $A \subseteq \mathcal{G} \sqcup \mathcal{R}_\Delta$ we denote with $Y_A$ the intersections of all divisors associated to $A$.
The graded differential algebra $\MM$ is the direct sum of vector spaces 
\[\bigoplus_{A \subset \mathcal{G} \sqcup \mathcal{R}_\Delta} H^\bigcdot (Y_A,\Q),\]
on which:
\begin{itemize}
 
\item the total degree of the elements in $H^p(Y_\A)$ is $\lvert A \rvert + p$;
\item the multiplication is induced by the restriction maps and the cup product 
\[H^p (Y_A) \otimes H^{p'} (Y_B) \to H^{p+p'} (Y_{A\cup B});\]
\item the differential is defined from the Gysin map $H^p(Y_A) \to H^{p+2}(Y_{A\setminus \{a\}})$.   
\end{itemize}
The cohomology of each stratum $Y_A$ is computed in \cite[Theorem 9.1]{DCG2} in terms of some generators $t_W, s_j \in H^2(Y_\A)$.

We define a morphism $\tilde{f}\colon E \to \MM$ on generators by
\begin{align*}
    & s_W \mapsto 1 \in H^0(D_W), && t_W \mapsto t_W \in H^2(Y(\Delta, \mathcal{G})), \\
    & b_j \mapsto 1 \in H^0(D_j), && c_j \mapsto c_j \in H^2(Y(\Delta, \mathcal{G})).
\end{align*}

\begin{lemma}
The map $\tilde{f}$ is a surjective morphisms of differential graded algebras.
\end{lemma}

\begin{proof}
As shown in \cite[Theorem 9.1]{DCG2}, the restriction maps $H^\bigcdot(Y_A) \to H^\bigcdot(Y_B)$ for $A\subset B$ are surjective.
Since $\im \tilde{f}$ contains $H^\bigcdot(Y(\Delta,\mathcal{G}))$ and the elements $1 \in H^0(D)$ for all divisors $D$, the morphisms $\tilde{f}$ is surjective.
By construction of the cohomology algebra, the elements $t_W$ and $c_j$ of $H^2(Y(\Delta, \mathcal{G}))$ are $t_W = (i_W)_*(1)$ and $c_j = (i_j)_*(1)$, where $i_*$ is the Gysin morphism for the regular embedding $i\colon D \hookrightarrow Y(\Delta, \mathcal{G})$.
Therefore, $\tilde{f}$ is a morphism of differential graded algebras.
\end{proof}

The map $\tilde{f}$ factors trough $f \colon \DD \to \MM$, indeed we have the following theorem.

\begin{lemma} \label{thm:iso_Morgan}
The map $f$ is well defined and is an isomorphism.
\end{lemma}

\begin{proof}
We first check that \eqref{item:first_def_M}-\eqref{item:last_def_M} belong to $\ker \tilde{f}$:
\begin{enumerate}
    \item there are four cases to check:
\begin{itemize}
    \item $\tilde{f}(s_W b_j)$ is zero since $D_W$ and $D_j$ do not intersect for $j \not \in \Ann \Lambda_W$;
    \item $\tilde{f}(s_W c_j)$ is zero since $c_j =0$ in $H^\bigcdot (D_W) $ by \ref{item:c_wonderful};
    
    \item $\tilde{f}(t_W b_j)=t_W \in H^\bigcdot(D_j)$ is zero by \ref{item:t_wonderful};
    
    \item $t_W c_j=0$ in $H^\bigcdot(Y(\Delta,\mathcal{G}))$ by \ref{item:tc_wonderful}.
    
\end{itemize}

    \item we have $\tilde{f}(\prod_{W \in A} s_W \prod_{W \in B} t_W)= \prod_{W \in B} t_W=0$ in $H^\bigcdot (Y_A)$ by \ref{item:t_wonderful} since $A\cup B$ is not $\mathcal{G}$-nested.
    \item the element $\tilde{f}(\prod_{j \in A} b_j \prod_{j \in B} c_j)=\prod_{j \in B} c_j$ is zero in $H^\bigcdot(Y_A)$ by \ref{item:c_wonderful}.
    \item the vanishing of the linear relation follows from  \ref{item:linear_toric}.
    \item we have 
    \[\tilde{f}(F(A,W,B))= P_W^V \left( \sum_{L \in \mathcal{G}_{\geq W}} -t_L \right) \prod_{L \in B} t_L\] that is zero by \ref{item:P_wonderful}.
\end{enumerate}
We have proven that $f$ is well defined and surjective, since $\tilde{f}$ is.
Let $I$ be the ideal in $E$ generated by \eqref{item:first_def_M}-\eqref{item:last_def_M} and notice that $I$ is a monomial ideal in the variable $s_W$ and $b_j$ for $W \in \mathcal{G}$ and $j \in \mathcal{R}_\Delta$.
It is enough to prove that 
\[f\left(\prod_{W \in A} s_W \prod_{j\in B} b_j z\right)=0 \:\mbox{ implies }\: \prod_{W \in A} s_W \prod_{j\in B} b_j z=0\]
in $\DD$ for all subsets $A\subseteq \mathcal{G}$, $B\subseteq \mathcal{R}_\Delta$ and all polynomials $z$ in the variables $\{t_W\}_{W \in \mathcal{G}}$ and $\{c_j\}_{j \in \mathcal{R}_\Delta}$.

The monomials $\prod_{W \in A} s_W \prod_{j\in B}b_j$ with $A$ a non-nested belong to $I$ by \eqref{item:second_def_M}, the ones with $B$ not a cone belong to $I$ by \eqref{item:third_def_M}, and the ones with $B \not \subset \bigcap_{W \in A} \Ann \Lambda_W$ are in $I$ by \eqref{item:first_def_M}.

Now, let $A$ be a $\mathcal{G}$-nested set and $B\in \mathcal{C}_\Delta$ be a cone contained in $\bigcap_{W \in A} \Ann \Lambda_W$.
We define a map $H^\bigcdot(Y_{A\sqcup B}) \to \DD$ by using the presentation of \Cref{thm:DCG}: the morphism is defined by $z \mapsto \prod_{W \in A} s_W \prod_{j\in B} b_j z$ for all $z$ in the exterior algebra on generators $\{t_W\}_{W \in \mathcal{G}}$ and $\{c_j\}_{j \in \mathcal{R}_\Delta}$.
It is well defined:
\begin{enumerate} 
    \item[(T1)] holds by relation \eqref{item:third_def_M},
    \item[(T2)] holds by relation \eqref{item:four_def_M},
    \item[(W1)] holds by relations \eqref{item:first_def_M} and \eqref{item:third_def_M},
    \item[(W3a)] holds by relation \eqref{item:last_def_M},
    \item[(W3b)] holds by relations \eqref{item:second_def_M} and \eqref{item:first_def_M}.
\end{enumerate}
The composition
\[ H^\bigcdot(Y_{A\sqcup B}) \to \DD \to \MM \twoheadrightarrow H^\bigcdot(Y_{A\sqcup B})\]
is the identity, therefore if $f(\prod_{W \in A} s_W \prod_{j\in B} b_j z)=0$ then $z=0$ in $H^\bigcdot (Y_{A\cup B})$ and $\prod_{W \in A} s_W \prod_{j\in B} b_j z=0$ in $\DD$.
\end{proof}

\Cref{thm:iso_Morgan} togheter with the main result of \cite{Morgan78} imply the following result.

\begin{theorem}\label{dga-iso}
The differential graded algebra $(\DD, \dd)$ built in \Cref{def:D} is a model for the complement $M_\A$. 
Therefore,
$H^\bigcdot (\DD, \dd) \cong H^\bigcdot(M_A;\Q)$. \hfill \qedsymbol
\end{theorem}

\section{Divisorial case}
In this section we consider arrangements of subtori of codimension 1, usually known in the literature as \emph{toric arrangements}. Given such an arrangement  $\mathcal A=\{S_1,\dots,S_n\}$ we consider the toric wonderful model $Y(\Delta,\mathcal{G})$ where $\mathcal{G} = \LLL_{>\hat{0}}$ is the maximal building set.
In this case, the $\mathcal{G}$-nested subsets coincide with the chains in $\LLL_{>\hat{0}}$.

Inspired by Yuzvinsky \cite{Yu02,Yu99}, we introduce a different set of generators $\sigma_W, \tau_W$ for the d.g.a.\ $\DD$ and we determine the relations between them (\Cref{lemma:prop_base}).
By using this generators we define some elements $\Xi_{W,A}$ of $\D$ (\Cref{def:Xi}) that belongs to the kernel of $\dd$ (\Cref{lemma:in_ker}).
We study the multiplication between them in \Cref{lemma:multiplication} and their relation with the cohomology of the ambient torus (\Cref{lemma:facile}).
The linear relation between $\Xi_{W,A}$ are rather complicated to prove (\Cref{lemma:beta_b,cor:beta_cones,lemma:cones,lemma:circuit}).
In the main result of this section, \Cref{thm:main_div}, we introduce a Orlik-Solomon type algebra $R$ and we prove that the composition
\[ R \to H(\DD,\dd) \cong H(M(\mathcal{A});\Q) \]
is an isomorphism.
The map $R \to H(\DD,\dd)$ is well defined by all the Lemmas preceding the main Theorem, is injective by \Cref{lemma:monomial_base,lemma:nbc_basis,lemma:increasing_flag}, and surjective by dimensional argument (\Cref{lemma:nbc_basis}).

Although the d.g.a.\ $\DD$ depends on the choice of a good fan $\Delta$, the algebra $R$ and its isomorphic image in $\DD$ are independent from the choice of the fan.

In this section we will use basic notions of matroid theory such as those of \emph{independent set},\emph{circuit}, \emph{no broken circuit}, that can be found for instance in the paper \cite{OS80}.

Define the elements $\sigma_W = \sum_{L\geq W} s_L$ and $\tau_W = \sum_{L\geq W} t_L$ in $\DD$ for all $W \in \LLL_{\geq \hat{0}}$.
Moreover, for every $\chi \in \Lambda$ define \[\beta_\chi^-
= -\sum_{j \in \mathcal{R}_\Delta} \min(0, \langle \chi, j\rangle) b_j, \qquad \beta_\chi^+
= \sum_{j \in \mathcal{R}_\Delta} \max(0, \langle \chi, j\rangle) b_j\] and \[\beta = \beta^+ - \beta^-, \qquad \gamma_\chi^- = -
\sum_{j \in \mathcal{R}_\Delta} \min(0, \langle \chi, j\rangle) c_j.\]

As in the previous Section, we consider the bi-gradation of $\DD$ given by $\deg(s_W)=\deg(b_j)=(0,1)$ and $\deg(t_W)=\deg(c_j)=(2,0)$, so that the differential $\dd$ has bi-degree $(2,-1)$.

\begin{lemma}\label{lemma:monomial_base}
The set $\{\prod_{W \in A} s_W \prod_{j \in C} b_j\}_{A,C}$, where $A$ runs over all the $\mathcal{G}$-nested sets and $C \in \mathcal{C}_{\Delta}$ over all the cones contained in $\cap_{W \in A} \Ann \Lambda_W$, is a linear basis of $\DD^{0,\bigcdot}$.

Moreover, the set $\{\prod_{W \in A} \sigma_W \prod_{j \in C} b_j\}_{A,C}$ (where $A$ and $C$ runs over the range described above) is a linear basis of $\DD^{0,\bigcdot}$.
\end{lemma}
\begin{proof}
Notice that $\DD^{0,\bigcdot}$ is the exterior algebra on generators $s_W$ and $b_j$ with relations:
\begin{enumerate}[label=(\arabic*')]
\item $s_Wb_j$ if $j \not \in \Ann \Lambda_W$,
\item $\prod_{W \in A} s_W$ if $A$ is not $\mathcal{G}$-nested,
\item $\prod_{j \in C} b_j$ if $A$ is not a cone in $\Delta$. 
\end{enumerate}
These relations generate a monomial ideal, so can be easily seen that these element divide a monomial if and only if it is not in the first basis.

For the second basis, we choose a total order on the set $\mathcal{G}$ that refines the partial order on it.
This total order induces a lexicographical order on the set of $\mathcal{G}$-nested sets.
The matrix that represents the elements $\{\prod_{W \in A} \sigma_W \prod_{j \in C} b_j\}_{A,C}$ in the basis $\{\prod_{W \in A} s_W \prod_{j \in C} b_j\}_{A,C}$ is upper triangular with ones on the diagonal entries.
This proves the claim.
\end{proof}

\begin{lemma} \label{lemma:prop_base}
In $\DD$ we have the following relations:
\begin{enumerate}
\item $\sigma_W \sigma_L = (\sigma_W-\sigma_L)\left( \sum_{V \in W \vee L} \sigma_V \right)$ for all $W,L\in \LLL$,
\item if $\chi \in \Lambda_V$ then $x_\chi y_V=0$ where $x=\beta, \beta^-, \beta^+$ or $\gamma^-$ and $y= \sigma$ or $\tau$,
\item if $W \gtrdot V$ and $\chi \in \Lambda_W$ is an element that generates $\Lambda_W/\Lambda_V$, then:
\[
    \sigma_V (\tau_W+ \gamma^-_\chi)=\sigma_W \tau_W, \qquad
    \tau_V (\tau_W+ \gamma^-_\chi)=\tau_W^2.
\]
\end{enumerate}
\end{lemma}
\begin{proof}$ $
\begin{enumerate} 
\item Let $x_1=\sum_{\substack{V\geq W \\ V \not \geq L}}s_V$, $x_2=\sum_{\substack{V\geq L \\ V \not \geq W}}s_V$ and $x_3=\sum_{\substack{V\geq W \\ V \geq L}}s_V$.
The claimed equality can be rewritten as $(x_1+x_3)(x_2+x_3) = (x_1-x_2) x_3$.
Since $x_3$ has degree one we have $x_3^2=0$ and we need to prove that $x_1x_2=0$.
This follows from $x_1x_2 = \sum s_V s_U$ where the sum runs over all $V\geq W$, $V\not \geq L$ and $U\geq W$, $U \not \geq L$:
we have $s_V s_U=0$ because $V$ and $U$ do not form a chain.
\item Notice that for $\chi \in \Lambda_W$ we have $\min(0, \langle \chi, j \rangle) a_j r_W=0$ for $a=b$ or $a=c$ and $r=s$ or $r=t$, by Relation \eqref{item:first_def_M} of \Cref{def:D}.
Since $W\geq V$ implies $\Lambda_W \supseteq \Lambda_V$, we have 
\[x_\chi y_V= \sum_{\substack{W \geq V \\ j \in \mathcal{R}_\Delta}}\min(0, \langle \chi, j \rangle) a_j r_W =0.\]
\item 
If $L\geq W$, we have that $s_L(\tau_W+\gamma_\chi)=s_L \tau_W$ since $\chi \in \Lambda_L$. On the other hand, if $L\geq V$, $L\not \geq W$, we will show that
$s_L (\tau_W+ \gamma^-_\chi)=0$.
Indeed, we have 
\[s_L(\tau_W+\gamma_\chi^-) = \left( \sum_{U \in L \vee W} s_L \tau_U \right) + s_L \gamma_\chi^-.\]
Let $\eta \in \Lambda_U$ be an element that generates $\Lambda_U/\Lambda_L$ and notice that $\chi = a \eta + \eta'$ with $\eta' \in \Lambda_L$ and $a= \lvert L\vee W \rvert$.
Observe that
\begin{align*}
\gamma_\chi^- s_L &= \sum_{j \in \mathcal{R}_\Delta} -\min(0, \langle \chi,j \rangle) c_js_L \\
&= \sum_{\substack{j \in \mathcal{R}_\Delta \\ j \in \Ann \Lambda_L}} -\min(0, \langle a \eta + \eta',j \rangle) c_js_L \\
&= \sum_{\substack{j \in \mathcal{R}_\Delta \\ j \in \Ann \Lambda_L}} - a\min(0, \langle \eta,j \rangle) c_js_L \\
&= \sum_{j \in \mathcal{R}_\Delta } - a\min(0, \langle \eta,j \rangle) c_js_L =a \gamma_\eta^- s_L    
\end{align*}
and so \[s_L (\tau_W+ \gamma^-_\chi)= \sum_{U \in L \vee W} s_L (\tau_U+\gamma^-_\eta) = \sum_{U \in L \vee W} F(\{ L \},U, \emptyset)=0\] by Relation \eqref{item:last_def_M} of \Cref{def:D}.
The proof of $\tau_V (\tau_W+ \gamma^-_\chi)=\tau_W^2$ is analogous.
\qedhere
\end{enumerate}
\end{proof}

Let $A\subseteq \{1,\dots,n\}$ be an independent set and $W$ a connected component of $\cap_{a \in A} S_a$. Following \cite{MoT}, we denote by $m(A)$ the number of connected components in such intersection.
A \textit{flag} adapted to $A$ and $W$ is a list $\mathcal{F} = (a_1, a_2, \dots, a_k)$ of distinct elements of $A$; $m(\mathcal{F})$ is defined accordingly.
The list $\mathcal{F}$ determines a flag in the usual sense by setting $F_0=T$ and $F_{i}$ being the unique connected component of $S_{a_i} \cap F_{i-1}$ containing $W$.
Then $F_i$ is a flat $\forall i = 0, \dots, k$ and $F_{i} \lessdot F_{i+1}$.
Viceversa, every maximal flag $(F_0=T \lessdot F_1 \dots \lessdot F_k)$ between $T$ and $F_k$ with $F_k\leq W$ determines a unique flag $\mathcal{F} = (a_1, a_2, \dots, a_k)$ adapted to $A$ and $W$.

For every $i\in \{1,\dots,n\}$ we choose a character $\chi_i $ that generates $\Lambda_{S_i}$.
For each flag $\mathcal{F}$ and each $a\in A$, we define the elements $x(\mathcal{F},a)=\sigma_{F_i}$ if $a=a_i$ and $x(\mathcal{F},a)= \beta_{\chi_a}^-$ otherwise.
Analogously we set $y(\mathcal{F},a)=\tau_{F_i}$ if $a=a_i$ and $y(\mathcal{F},a)= \gamma_{\chi_a}^-$, otherwise.

\begin{definition} \label{def:Xi}
For each independent set $A \subseteq \{1,\dots,n\}$ and each connected component $W$ of $\cap_{a \in A} S_a$ we define the following element of $\DD$:
\[ \Xi_{W,A} = \sum_{\mathcal{F}} 
 \frac{m(\mathcal{F})}{m(A)} \prod_{a\in A} x(\mathcal{F},a),\]
where the sum is taken over all the flags adapted to $A$ and $W$.
\end{definition}

In order to simplify the notations, in the definition above we denoted by $\prod_{a\in A}$ the exterior product taken in the order of $A\subseteq \{1,\dots,n\}$. The same notation will be used from now on.

\begin{lemma}\label{lemma:in_ker}
For each independent set $A \subseteq \{1,\dots,n\}$ and each connected component $W$ of $\cap_{a \in A} S_a$ we have $\dd(\Xi_{W,A})=0$.
\end{lemma}
\begin{proof}
We have that 
\[\dd(\Xi_{W,A}) = \sum_{\mathcal{F}} 
\frac{m(\mathcal{F})}{m(A)} \sum_{b \in A} (-1)^{\lvert A_{<b} \rvert} y(\mathcal{F},b)\prod_{a \in A \setminus \{b\}} x(\mathcal{F},a),\]
so define $Z(\mathcal{F},b)=y(\mathcal{F},b)\prod_{a \in A \setminus \{b\}} x(\mathcal{F},a)$.

If $k>1$, then $Z(\mathcal{F},a_1)=0$ because 
\[\tau_{F_1} \sigma_{F_2} = (\tau_{F_1}+ \gamma^-_{\chi_{a_1}}) \sigma_{F_2}=F(\emptyset, F_1, \emptyset)\sigma_{F_2} =0.\]
If $i \neq 1,k$, then 
\[Z(\mathcal{F},a_i)= 
\sigma_{F_1}\sigma_{F_2} \dots \sigma_{F_{i-2}} \sigma_{F_i}\tau_{F_i} \sigma_{F_{i+1}} \dots \sigma_{F_k} \prod_{a \in A \setminus \mathcal{F}} \beta_{\chi_a}^-,\]
because $\sigma_{F_{i-1}}\tau_{F_i}\sigma_{F_{i+1}}= \sigma_{F_{i-1}}(\tau_{F_i}+\gamma^-_{\chi_{a_{i}}}) \sigma_{F_{i+1}}=\sigma_{F_i} \tau_{F_i} \sigma_{F_{i+1}}$.
Moreover for $i \neq 1,k$, $\mathcal{F} = (a_1, a_2, \dots, a_k)$, we consider the flag $\mathcal{F}'=(a_1, \dots, a_{i-2}, a_i, a_{i-1}, a_{i+1}, \dots, a_k)$ and notice that 
\[ Z(\mathcal{F}',a_{i-1})= (-1)^{\lvert A_{<a_i}\rvert - \lvert A_{<a_{i-1}} \rvert - 1} Z(\mathcal{F},a_i)\]

If $k>0$, then:
\[ Z(\mathcal{F},a_k)+ \frac{m(\mathcal{F}\setminus a_k)} {m(\mathcal{F})} Z(\mathcal{F}\setminus a_k, a_k)=
\sigma_{F_1}\sigma_{F_2} \dots \sigma_{F_{k-2}} \sigma_{F_k}\tau_{F_k} \prod_{a \in A \setminus \mathcal{F}} \beta_{\chi_a}^-, \]
because $\sigma_{F_{k-1}}(\tau_{F_k}+ \frac{m(\mathcal{F}\setminus a_k)} {m(\mathcal{F})} \gamma^-_{\chi_{a_{k}}})= \sigma_{F_{k-1}}(\tau_{F_k}+ \gamma^-_{\chi}) = \sigma_{F_k} \tau_{F_k}$, where $\chi\in \Lambda_{F_k}$ is any element such that $m(\mathcal{F}) \chi - m(\mathcal{F}\setminus a_k) \chi_{a_{k}} \in \Lambda_{F_{k-1}}$.
Moreover, for $\mathcal{F} = (a_1, a_2, \dots, a_k)$ we consider the flag $\mathcal{F}'=(a_1, \dots, a_{k-2}, a_k, a_{k-1})$ and we have:
\begin{multline*}
Z(\mathcal{F'},a_{k-1})+ \frac{m(\mathcal{F'}\setminus a_{k-1})} {m(\mathcal{F})} Z(\mathcal{F'}\setminus a_{k-1}, a_{k-1})= \\ = (-1)^{\lvert A_{<a_i}\rvert - \lvert A_{<a_{i-1}} \rvert - 1} \left( Z(\mathcal{F},a_k)+ \frac{m(\mathcal{F}\setminus a_k)} {m(\mathcal{F})} Z(\mathcal{F}\setminus a_k, a_k) \right)
\end{multline*}

Finally, we have:
\begin{align*}
    &m(A)\dd(\Xi_{W,A}) = \sum_{\mathcal{F}} 
m(\mathcal{F}) \sum_{b \in A} (-1)^{\lvert A_{<b} \rvert} Z(\mathcal{F},b) \\
 &= \sum_{ \lvert \mathcal{F} \rvert > 1} \sum_{i=1}^{k} (-1)^{\lvert A_{<a_i} \rvert}  m(\mathcal{F})  \sigma_{F_1} \dots \sigma_{F_{i-2}} \sigma_{F_i}\tau_{F_i} \sigma_{F_{i+1}} \dots \sigma_{F_k} \prod_{a \in A \setminus \mathcal{F}} \beta_{\chi_a}^- \\
  &= \sum_{i=2}^{\lvert A \rvert} \sum_{\lvert \mathcal{F} \rvert \geq i} (-1)^{\lvert A_{<a_i}\rvert} m(\mathcal{F})  \sigma_{F_1} \dots \sigma_{F_{i-2}} \sigma_{F_i}\tau_{F_i} \sigma_{F_{i+1}} \dots \sigma_{F_k} \prod_{a \in A \setminus \mathcal{F}} \beta_{\chi_a}^- \\
 &= \sum_{i=2}^{\lvert A \rvert} \sum_{\substack{\lvert \mathcal{F} \rvert \geq i \\ a_i>a_{i-1}}} (1-(-1)^{1}) m(\mathcal{F})  \sigma_{F_1} \dots \sigma_{F_{i-2}} \sigma_{F_i}\tau_{F_i} \sigma_{F_{i+1}} \dots \sigma_{F_k} \prod_{a \in A \setminus \mathcal{F}} \beta_{\chi_a}^- \\
    &= 0,
\end{align*}
where all the products are taken in the order of $A$ with $\tau_{F_i}$ in position of $a_i$, $\sigma_{F_i}$ in position $a_{i-1}$ and $\sigma_{F_j}$ in position $a_j$.
This completes the proof.
\end{proof}

We recall that, given two positive integers $k, h$, a \emph{$(k,h)$-shuffle} is an element $p$ of the symmetric group on the elements $\{1,\dots, k+h\}$ such that $p(i)<p(j)$ for every couple $i<j$ such that either $i,j\in\{1,\dots,k\}$, or $i,j\in\{k+1,\dots,k+h\}$.

\begin{lemma} \label{lemma:multiplication}
For all independent set $A$ and $B$, and for all connected components $W$ of $\cap_{a \in A} S_a$ and $L$ of $\cap_{b \in B} S_b$, we have
$\Xi_{W,A}\Xi_{L,B} =0$ if $A\cap B$ is not empty or if $A \sqcup B$ is not independent.
Otherwise
\[\Xi_{W,A}\Xi_{L,B} = (-1)^{l(A,B)} \sum_{V \in W \vee L} \Xi_{V,A\cup B}, \]
where $l(A,B)$ is the sign of the permutation reordering $(A,B)$.
\end{lemma}

\begin{proof}
Let $\mathcal{F}=(a_1,\dots, a_k) $ and $\mathcal{G}=(a_{k+1},\dots, a_{k+h})$ be two flags and let $A=\{ a_1,\dots, a_{k+h} \}$. 
If $A$ is independent of cardinality $k+h$, then for each $(k,h)$-shuffle $p$ and each element $V \in F_k \vee G_h$ we have a flag $\mathcal{F}*_p \mathcal{G} := (a_{p(1)},\dots, a_{p(k+h)})$ adapted to $A$ and $V$.
By using only equation~(1) of \Cref{lemma:prop_base}, we have
\[
\prod_{i=1}^k \sigma_{F_i} \prod_{j=1}^h \sigma_{G_j} =
\sum_{V \in F_k \vee G_h} \sum_{p \textnormal{ shuffle}} 
 \prod_{i=1}^{k+h} \sigma_{(F*_p G)_i}.
\]
The number of connected components of $W\cap L$ contained in $V$ is equal to 
\[\frac{m(A\cup B)}{m(A)m(B)} \frac{m(\mathcal{F}) m(\mathcal{G}) }{m(\mathcal{F}\cup \mathcal{G})}.\]
Finally,
\begin{align*}
\Xi_{W,A}\Xi_{L,B} &= (-1)^{l(A,B)} \sum_{\mathcal{F},\mathcal{G}} \frac{m(\mathcal{F})m(\mathcal{G})}{m(A)m(B)} \prod_{a \in A} x(\mathcal{F},a) \prod_{b \in B} x(\mathcal{G},b)\\
&= (-1)^{l(A,B)} \sum_{\mathcal{F},\mathcal{G}} \frac{m(\mathcal{F})m(\mathcal{G})}{m(A)m(B)} \sum_{V \in F_k \vee G_h} \sum_{p \textnormal{ shuffle}} \prod_{a \in A \sqcup B} x(\mathcal{F}*_p\mathcal{G},a) \\
&= (-1)^{l(A,B)} \sum_{\mathcal{F},\mathcal{G}} \sum_{V \in W \vee L} \sum_{p \textnormal{ shuffle}} \prod_{a \in A \sqcup B} x(\mathcal{F}*_p\mathcal{G},a) \\
&= (-1)^{l(A,B)} \sum_{V \in W \vee L} \sum_{\mathcal{F},\mathcal{G}}  \sum_{p \textnormal{ shuffle}} \prod_{a \in A \sqcup B} x(\mathcal{F}*_p\mathcal{G},a) \\
&= (-1)^{l(A,B)} \sum_{V \in W \vee L} \sum_{\mathcal{H}} \prod_{a \in A \sqcup B} x(\mathcal{H},a) \\
&= (-1)^{l(A,B)} \sum_{V \in W \vee L} \Xi_{V,A\cup B},
\end{align*}
where we used the fact that flags $\mathcal{H}$ adapted to $A\sqcup B$ and $V$ are in bijection with flags $\mathcal{F}*_p \mathcal{G}$ where $p$ runs over all the $(k,h)$-shuffles, $\mathcal{F}$ over all flags adapted to $A$ and $W$, and $\mathcal{G}$ over all flags adapted to $B$ and $L$.
So the claim follows.
\end{proof}

\begin{lemma}\label{lemma:facile}
If $\chi \in \Lambda_W$, then $\Xi_{W,A}\beta_\chi=0$.
\end{lemma}
\begin{proof}
Let $\mathcal{F}$ be a flag adapted to $A$ and $W$, we show that $\beta_\chi \prod_{a\in A} x(\mathcal{F},a)  = 0$.

Notice that $\prod_{i=0}^k \beta_{\chi_i}=0$ if the elements $\chi_{i}$, $i=0, \dots, k$ are linearly dependent because $k+1$ distinct rays in a $k$-dimensional vector space cannot span a simplicial cone.

Let $L$ be a layer contained in $F_k$, if $\chi \in \Lambda_L$ then $\beta_\chi s_L=0$ and hence $\beta_\chi \sigma_{F_k}=0$.
If $\chi \not \in \Lambda_{F_k}$, then $\chi$ and $\chi_a$ for $a \in A \setminus \mathcal{F}$ are nonzero and dependent in $\Lambda_W/\Lambda_{F_k}$ therefore $s_L \beta_\chi \prod_{a \in A \setminus \mathcal{F}} \beta_{\chi_a} = 0$.

Finally we have $\beta_\chi \sigma_{F_k} \prod_{a \in A \setminus \mathcal{F}} \beta_{\chi_a} = 0$ and so the claim holds.
\end{proof}

\begin{lemma} \label{lemma:beta_b}
Let $A$ be an independent set and $s_a\in \{ +,-\}$ for $a \in A$. Let $Z$ be the set $\{v \in \Lambda^* \mid \langle v, s_a \chi_a \rangle >0 \textnormal{ for all } a \in A \}$.
Consider the projection $\pi \colon \Lambda^* \to \Lambda^*/\Ann \Lambda_A$.
We have that
\[ \prod_{a \in A} \beta_{\chi_a}^{s_a} = m(A) \sum_{ \substack{ K \in \mathcal{C}_\Delta^{\lvert A\rvert} \\ K \subset Z}} \Vol (\pi(K)) \prod_{c \in K} b_c,\]
where the last product is taken in any order such that the two bases $(s_a \chi_a)_{a \in A}$ and $(\pi(c))_{c \in K}$ are both positive or both negative.
\end{lemma}
\begin{proof}
Let $K\in \mathcal{C}_\Delta^{\lvert A\rvert}$ be a $\lvert A \rvert$-dimensional cone not contained in $Z$: then there exists $c\in K$ such that $c \not \in Z$.
So, for some $a \in A$, we have $\min (0, \langle s_a \chi_a, c' \rangle)=0$ for all $c' \in K$ by using the equal sign property.
It easy to see that the monomial $\prod_{c \in K} b_c$ does not appear in $\prod_{a \in A} \beta_{\chi_a}^{s_a}$.

Now suppose that $K=(k_1, \dots, k_l)\in \mathcal{C}^l_\Delta$ is contained in $Z$, the coefficient of $\prod_{i=1}^{l} b_{k_i}$ in $\prod_{a \in A} \beta_{\chi_a}^{s_a}$ is 
\[\sum_{\sigma \in \SG_l} \prod_{i=1}^k (-1)^{\sgn \sigma} \langle s_i\chi_i, k_{\sigma(i)} \rangle.\]
Now notice that $\langle s_i\chi_i, k_{\sigma(i)} \rangle = \langle s_i\chi_i, \pi(k_{\sigma(i)}) \rangle$ for all $i$ and $\sigma$.

The equality 
\[\sum_{\sigma \in \SG_l} \prod_{i=1}^l (-1)^{\sgn \sigma} \langle s_i\chi_i, \pi(k_{\sigma(i)}) \rangle = \det(s_i\chi_i) \det(\pi(k_{i}))\] follows from the multilinearity in the entries $s_i\chi_i$ and $\pi(k_{i})$.
Since the two bases $(s_a \chi_a)_{a \in A}$ and $(\pi(k))_{k \in K}$ are both positive (resp.\ negative) then $\det(s_i\chi_i) \det(\pi(k_{i}))$ is positive and equals to $m(A) \Vol (\pi(K))$.
\end{proof}

The proof of this corollary follows from the proof of \Cref{lemma:beta_b} by omitting some steps.
\begin{corollary}\label{cor:beta_cones}
Let $A$ be an independent set, then:
\[ \prod_{a \in A} \beta_{\chi_a} = m(A) \sum_{ K \in \mathcal{C}_\Delta^{\lvert A\rvert}} \Vol (\pi(K)) \prod_{c \in K} b_c,\]
where the last product is taken in any order such that the two bases $(\chi_a)_{a \in A}$ and $(\pi(c))_{c \in K}$ are both positive or both negative. \hfill \qedsymbol
\end{corollary}

\begin{lemma}\label{lemma:cones}
Let $X$ be a subset such that $\lvert X \rvert = \rk(X)+1$, $C\subseteq X$ be the unique circuit, $A \subset X$ be a independent set, $F$ be a connected component of $\cap_{a \in A} S_a$, and $j$ be an element of $C\setminus A$.
There exists a minimal relation $\sum_{i \in C} c_i m(C \setminus \{i\})\chi_i =0$ for some $c_i \in \{+,-\}$. 
Suppose that $C':=C \setminus A$ has cardinality at least $2$, then
\[\sigma_F \sum_{j \in C'} \frac{(-1)^{\lvert C'_{<j}\lvert}}{m(X \setminus \{j\})} \prod_{i \in C'\setminus \{j\}} \beta_{\chi_i}^{\delta(i,j)} =0,\]
where $\delta(i,j)=c_i c_j$ if $i<j$ and $\delta(i,j)=-$ if $i>j$.
\end{lemma}

\begin{proof}
For the sake of simplifying the notation, let us suppose $C'=\{0,1, \dots, l\}$.
The first step of the proof is to reduce to the case $c_i=-$ for $i<k$ and $c_i=+$ for $i \geq k$ for some $k\in C'$.
Let $\mu \in \SG_{\lvert C' \rvert}$ be the unique shuffle that reorders $C'$ in such a way that $c_i=-$ for $i<k$ and $c_i=+$ for $i \geq k$.
We have 
\[\sum_{j \in C'} \frac{(-1)^j}{m(X \setminus \{j\})} \prod_{i \in C'\setminus \{j\}} \beta_{\chi_i}^{\delta(i,j)} =
\sgn (\mu) \sum_{j \in C'} \frac{(-1)^{\mu(j)}}{m(X \setminus \{j\})} \prod_{i \in \mu(C'\setminus \{j\})} \beta_{\chi_i}^{\delta(i,\mu(j))}\]
where we use $\sgn(\mu)=(-1)^{j-\mu(j)}\sgn(\mu_{|C'\setminus\{j\}})$.
Moreover notice that $\delta(i,j)=\delta(\mu(i),\mu(j))$ since $(i,j)$ is an inversion of $\mu$ only if $c_ic_j=-$. Thus from now on we assume $c_i=-$ for $i<k$ and $c_i=+$ for $i \geq k$.

Define $Z_j= (\cap_{i < j} H_{\chi_i}^{c_ic_j} \cap_{i>j} H_{\chi_i}^{-})/ \Ann \Lambda_F$, $X_j=Z_j \cap H_{\chi_j}^+$ and $Y_j=Z_j \cap H_{\chi_j}^-$.
The following properties follows easily from the definition:
\begin{align*}
    & Z_j=X_j \cup Y_j & & \dim (X_j \cap Y_j) < l \\
    & X_l= \emptyset & & Y_{k-1}=\emptyset \\
    & X_k=Y_0 & & X_{j-1}=Y_j \textnormal{ for all } j\neq k
\end{align*}

By \Cref{lemma:beta_b} we have 
\begin{align*}
 \frac{ \sigma_F }{m(X \setminus \{j\})} \prod_{i \in C'\setminus \{j\}} \beta_i^{\delta(i,j)} &= \sigma_F  \frac{m(C' \setminus \{j\})}{m(X \setminus \{j\})}\sum_{ \substack{ K \in \mathcal{C}_\Delta^l \\ K \subset Z_j}} \Vol (\pi(K)) \prod_{c \in K} b_c \\
&= \sigma_F \frac{m(C')}{m(X)} \sum_{ \substack{ K \in \mathcal{C}_\Delta^l \\ K \subset Z_j}} \Vol (\pi(K)) \prod_{c \in K} b_c,
\end{align*}
where in the last equality we used the property (P) of \emph{arithmetic matroids} (see \cite{BrMo}).
For $j\neq k$ the bases $(\delta(i,j)\chi_i)_{i \neq j}$ and $(\delta(i,j-1)\chi_i)_{i \neq j-1}$ have the same orientation.
The bases $(-\chi_i)_{i >0}$ and $(\delta(i,k)\chi_i)_{i \neq k}$ have the same orientation if and only if $(-1)^{k-1}=1$.

Since
\[
    \sigma_F \sum_{j \in C'} \frac{(-1)^{\lvert C'_{<j}\lvert}}{m(X \setminus \{j\})} \prod_{i \in C'\setminus \{j\}} \beta_i^{\delta(i,j)} =
     \frac{m(C')}{m(X)} \sigma_F  \sum_{j \in C'}  \sum_{ \substack{ K \in \mathcal{C}_\Delta^l \\ K \subset Z_j}} (-1)^j \Vol (\pi(K)) \prod_{c \in K} b_c,
\]
it is enough to consider the following:
\begin{align*}
    \sum_{j \in C'}  &\sum_{ K\subset Z_j} (-1)^j \Vol (\pi(K)) \prod_{c \in K} b_c = 
    \!\begin{multlined}[t]
        \sum_{j \in C'}  \sum_{ K \subset X_j} (-1)^j \Vol (\pi(K)) \prod_{c \in K} b_c + \\
    + \sum_{j \in C'}  \sum_{ K \subset Y_j} (-1)^j \Vol (\pi(K)) \prod_{c \in K} b_c
    \end{multlined} \\
    & = \sum_{ K \subset X_k} (-1)^k \Vol (\pi(K)) \prod_{c \in K} b_c + \sum_{ K \subset Y_0} \Vol (\pi(K)) \prod_{c \in K} b_c =0
\end{align*}
so the claim follows.
\end{proof}

Let $C\subseteq \{1,\dots,n\}$ be a circuit oriented by the signs $(c_i)_{i \in C}$. We recall the following definition, which was introduced by Postnikov in \cite{Po}. For each $A\subseteq \{1,\dots,n\}$, we say that $C/A$ is a \emph{positroid} if $c_i=c_j$ for all $i,j \in C \setminus A$.

\begin{lemma}\label{lemma:circuit}
Consider $X\subseteq \{1,\dots,n\}$ such that $\lvert X \rvert = \rk(X)+1$, let $C\subseteq X$ be the unique circuit and $L$ be a connected component of $\cap_{i \in X} S_i$.
There exists a minimal relation $\sum_{i \in C} c_i m(C \setminus \{i\})\chi_i =0$ for some $c_i \in \{+,-\}$.
Then, we have
\[ \sum_{\substack{X\setminus C \subseteq A \subsetneq X\\ C/A \textnormal{ positroids}}} (-1)^{\lvert X_{<j} \rvert + l(A,B)} \frac{m(A)}{m(X\setminus \{j\})} \Xi_{W,A} \beta_B = 0 \]
where $j=\max (X \setminus A)$, $B=C \setminus (A \cup \{j\})$, $W$ is the connected component of $\cap_{a \in A} S_a$ containing $L$ and $l(A,B)$ is the sign of the permutation that reorders $(A,B)$.
\end{lemma}
\begin{proof}
We may assume that $X=\{0,1,\dots, \rk(X)\}$ and $C=\{0,1,\dots, \rk(C)\}$.
Let $R=X\setminus C$, we can rewrite the left hand side as follow:
\begin{align*}
\sum_{\substack{R \subseteq A \subsetneq X \\ C/A \textnormal{ pos.}}} &(-1)^{j+l(A,B)} \frac{m(A)}{m(X\setminus \{j\})} \Xi_{W,A} \beta_B =\\ 
&=\sum_{\mathcal{F} \subsetneq X} \sum_{\substack{\mathcal{F} \cup R\subseteq A \subsetneq X \\ C/A \textnormal{ pos.}}}
(-1)^{j
} \frac{m(\mathcal{F})}{m(X\setminus \{j\})} \prod_{i=1}^{\lvert \mathcal{F} \rvert} \sigma_{F_i} \prod_{a \in A\setminus \mathcal{F}} \beta_a^{-} \prod_{b \in B} \beta_b \\
&= \sum_{\mathcal{F} \subsetneq X} \sum_{\substack{D \subsetneq C\setminus \mathcal{F} \\ C/(D\cup \mathcal{F}) \textnormal{ pos.}} } 
(-1)^{j
} \frac{m(\mathcal{F})}{m(X\setminus \{j\})} \prod_{i=1}^{\lvert \mathcal{F} \rvert} \sigma_{F_i} \prod_{a \in D \cup (R\setminus \mathcal{F})} \beta_a^{-} \prod_{b \in B} \beta_b
\end{align*} 
Let $C'=C \setminus \mathcal{F}$, $j= \max(C\setminus A)$ and $C(j)=\{i \in C_{<j} \setminus \mathcal{F} \mid c_i=c_j\}$, we need the following equality:
\begin{align*}
\sum_{\substack{B \subseteq C_{<j} \setminus \mathcal{F} \\ B\cup \{j\} \textnormal{ pos.}}} &\prod_{a \in C'\setminus (B \cup \{j\}) } \beta_a^{-} \prod_{b \in B} \beta_b = \\
&= \sum_{B \subseteq C(j)} \sum_{D \subseteq B} (-1)^{ \lvert B \setminus D \rvert }
 \prod_{a \in C' \setminus (D \cup \{j\})} \beta_a^{-} \prod_{b \in D} \beta_b^+\\
 &=  \sum_{D \subseteq C(j)} 
 \prod_{a \in C' \setminus (D \cup \{j\})} \beta_a^{-} \prod_{b \in D} \beta_b^+ \sum_{E \subseteq C(j) \setminus D} (-1)^{ \lvert E \rvert } \\
 &= \prod_{a \in C' \setminus (C(j)\cup \{j\})} \beta_a^{-} \prod_{b \in C(j)} \beta_b^+ \\
 &= \prod_{\substack{a =j+1,\dots, \rk(C) \\ a \not \in \mathcal{F}}} \beta_a^{-} \prod_{\substack{b = 0, \dots, j \\ b\not \in \mathcal{F}}} \beta_b^{c_b c_j}.
\end{align*}
We also need, for $\lvert C'\rvert >1$ the following:
\begin{align*}
\sigma_{F_k} \sum_{j\in C'} &  \frac{ (-1)^{\lvert C'_{<j} \rvert} }{m(X\setminus \{j\})}
\sum_{\substack{B \subseteq C'_{<j} \\ B\cup \{j\} \textnormal{ pos.}}} \prod_{a \in C'\setminus (B \cup \{j\}) } \beta_a^{-} \prod_{b \in B} \beta_b =\\
&= \sigma_{F_k} \sum_{j\in C'}  \frac{(-1)^{\lvert C'_{<j} \rvert}}{m(X\setminus \{j\})}
\prod_{a =j+1}^{\rk(C)} \beta_a^{-} \prod_{b = 0}^{j} \beta_b^{c_b c_j} \\
&= 0,
\end{align*}
by \Cref{lemma:cones}.
Finally, we have:
\begin{align*}
\sum_{\substack{R \subseteq A \subsetneq X \\ C/A \textnormal{ pos.}}} &(-1)^{j} \frac{m(A)}{m(X\setminus \{j\})} \Xi_{W,A} \beta_B = \\
&= \sum_{j\in C} \sum_{C \setminus \{j \} \subseteq \mathcal{F} \subsetneq X} (-1)^{j} \frac{m(\mathcal{F})}{m(X\setminus \{j\})} \prod_{i=1}^{\lvert \mathcal{F} \rvert} \sigma_{F_k} \prod_{a \in X \setminus (\mathcal{F} \cup \{j\})} \beta_a^- \\
&= \sum_{j\in C} \sum_{C \setminus \{j \} \subseteq \mathcal{F} \subsetneq X} (-1)^{j} \frac{m(\mathcal{F} \cup C)}{m(X)} \prod_{i=1}^{\lvert \mathcal{F} \rvert} \sigma_{F_k} \prod_{a \in X \setminus (\mathcal{F} \cup \{j\})} \beta_a^- \\
&=\sum_{j(\mathcal{F}) < k(\mathcal{F})} ((-1)^{j(\mathcal{F})}+(-1)^{j(\mathcal{F})-1}) \frac{m(\mathcal{F} \cup C)}{m(X)} \prod_{i=1}^{\lvert \mathcal{F} \rvert} \sigma_{F_k} \prod_{a \in X \setminus (\mathcal{F} \cup \{j\})} \beta_a^- \\
&=0,
\end{align*}
where $k(\mathcal{F})$ is the last element of $\mathcal{F}$ that belongs to $C$, $j(\mathcal{F})$ the unique element in $C\setminus \mathcal{F}$.
Let $\tilde{\mathcal{F}}$ be the flag obtained from $\mathcal{F}$ substituting $k(\mathcal{F})$ with $j(\mathcal{F})$.
Notice that $\mathcal{F}\cup C = \tilde{\mathcal{F}} \cup C$ and that the addendum associated to $\mathcal{F}$ and to $\tilde{\mathcal{F}}$ differ by the sign $(-1)^{k(\mathcal{F})-j(\mathcal{F})-1}$.
This prove the claimed equality.
\end{proof}

Let $\omega$ be the generator of $H^1(\C^*;\Z)$.

\begin{theorem} \label{thm:main_div}
	Let $\A$ be a toric arrangement.
	The rational cohomology algebra $H^*(M(\A);\mathbb Q)$ is isomorphic to the algebra 
	\[\faktor{H^\bigcdot(T;\Q)[e_{W,A}]}{I} \]
     where $A$ ranges over all the idependent subsets of $\{1,\dots,n\}$ and $W$ ranges over all connected components of $\cap_{a \in A} S_a$. The degree of the generator $e_{W,A}$ is $\vert A \vert$.
    The ideal $I$ is generated by the following elements:
	\begin{itemize}
		\item for any two generators $e_{W,A} $, $e_{W',A'}$,  
			\[e_{W,A} e_{W',A'}\]
		if $ A \cap  A' \neq \emptyset$ or $ A \sqcup A'  $ is a dependent set, and otherwise 
		\begin{equation}\label{eq:relazione_prodotto}
			e_{W,A}e_{W',A'}- (-1)^{l(A, A')} \sum_{L\in \pi_0(W\cap W')} e_{L,A\cup A'}.
		\end{equation}
		\item For any $\psi \in H^\bigcdot(T)$ such that $\psi_{|W}=0$,
		\begin{equation} \label{eq:restriction}
		    e_{W,A}\psi
		\end{equation}
        \item  For every $X\subseteq \{1,\dots,n\}$ such that $\rk(X)=\vert X \vert -1$ write $X=C\sqcup F$ with $C$ the unique circuit in $X$. Consider the  associated linear dependency $\sum_{i\in C} n_i\chi_i=0$ with $n_i\in \mathbb{Z}$, and for every connected component $L$ of $\cap_{i \in X} H_i$ a relation
\begin{equation}\label{eq:relazione_final}
 \sum_{\substack{X\setminus C \subseteq A \subsetneq X\\ C/A \textnormal{ positroids}}} (-1)^{\lvert X_{<j} \rvert+l(A,B)} \frac{m(A)}{m(X\setminus \{j\})} e_{W,A} \psi_B
\end{equation}
			where $j=\min(C\setminus A)$, $B=C \setminus (A \cup \{j\})$ and $\psi_B= \prod_{b \in B} \chi_b^*(\omega)$ an element in $H^\bigcdot(T)$.
		\end{itemize}
\end{theorem}

Before proving the above theorem we need a couple of lemmas.
We denote the ring $H^\bigcdot(T;\Q)[e_{W,A}]/I$ by $R$.

\begin{lemma}\label{lemma:nbc_basis}
There exists a filtration $\mathrm{F}_\bigcdot$ of $H^\bigcdot(T;\Q)[e_{W,A}]/I$ such that 
\[\gr_{\mathrm{F}}^\bigcdot R \cong \bigoplus_{W \in \mathcal{L}} H^\bigcdot (W) \otimes \tH_{\cd W-2}(\Delta(T,W)). \]
In particular the set $e_{W,A}$ with $A$ a no broken circuit set in $\mathcal{L}_{\leq W}$ generates $R$ as $H(T)$-module.
Moreover, $R$ and $H^\bigcdot(M(\A);\Q)$ have the same dimension.
\end{lemma}

\begin{proof}
Let $\mathrm{F}_\bigcdot$ be the filtration defined by \[\mathrm{F}_h R= \sum_{\substack{\cd(W)\leq h \\ A}} e_{W,A}H^\bigcdot (T).\]
The graded ring $\gr_{\mathrm{F}} R$ is isomorphic to $H^\bigcdot(T;\Q)[e_{W,A}]/I'$, where $I'$ is the ideal generated by eq.\ \eqref{eq:relazione_prodotto}, \eqref{eq:restriction} and
\begin{equation}
\sum_{j \in C} (-1)^{\lvert X_{<j}\rvert} e_{L,X \setminus \{j\}}
\tag{\ref{eq:relazione_final}'} \label{eq:OS}
\end{equation}
for all $X$ such that $\rk(X)=\lvert X \rvert -1$ and all $L$ connected components of $\cap_{a \in X} S_a$.
Notice that $\gr_{\mathrm{F}} R$ is $\mathcal{L}$-graded and isomorphic to 
\[
\bigoplus_{W \in \mathcal{L}} \faktor{H(W)[e_{W,A}]_A}{I_W}
 \]
as $H(T)$-module, where $I_W$ is the ideal generated by the eq.\ \eqref{eq:OS} for all $X$ such that $\rk(X)=\lvert X \rvert -1$ and $W$ is a connected components of $\cap_{a \in X} S_a$.
Finally, we have
\begin{align*}
\faktor{H^\bigcdot(T;\Q)[e_{W,A}]}{I} & \cong \gr_{\mathrm{F}} R \\
&\cong \bigoplus_{W \in \mathcal{L}} \faktor{H(W)[e_{W,A}]_A}{I_W} \\
&\cong \bigoplus_{W \in \mathcal{L}} H^\bigcdot (W) \otimes \tH_{\rk W-2}(\Delta(T,W)),
\end{align*}
where we use the Brieskorn isomorphism for the Orlik-Solomon algebra associated to the geometric lattice $\mathcal{L}_{\leq W}$.

From \Cref{thm:main_additive} we deduce that $R\cong H^\bigcdot(M(\A);\Q)$ as $\Q$-vector space and so they have the same dimension.
\end{proof}

We want to construct a bijection for any geometric lattice $\mathcal{L}_{\leq W}$ between \emph{no broken circuit sets} and certain maximal flags.
For any maximal flag of layers $\mathcal{F}=(T=F_0 \lessdot F_1 \lessdot \dots \lessdot F_k=W)$ we define the \textit{edge labelling} $\epsilon(\mathcal{F})$ as the list $(b_1, \dots, b_k)$ where $b_k= \max \{i \in \{1,\dots,n\} \mid F_k \in F_{k-1} \vee S_{i}\}$.
We say that $\mathcal{F}$ is \textit{increasing} if $b_i<b_j$ for all $i<j$ (where $\epsilon(\mathcal{F})=(b_1, \dots, b_k)$).

Notice that if $\mathcal{F}$ is a maximal flag adapted to $A$ and $W$, $\epsilon(\mathcal{F})$ may not be a subset of $A$.

\begin{lemma} \label{lemma:increasing_flag}
We fix a layer $W$ of rank $k$ and consider the geometric lattice $\mathcal{L}_{\leq W}$.
If $A=\{a_1 < a_2 < \dots < a_k\}$ is a no broken circuit set, then a maximal flag $\mathcal{F}$ adapted to $A$ and $W$ is increasing in $\mathcal{L}_{\leq W}$ if and only if $\mathcal{F}=(a_1, a_2, \dots,a_k)$.
\end{lemma}
\begin{proof}
The key observation is the following: if $b> a_k$ then $A \cup \{b\}$ is an independent set (since $A$ is a no broken circuit set).
We prove that every maximal, increasing flag adapted to $A$ and $W$ is $\mathcal{F}=(a_1, a_2, \dots,a_k)=\epsilon(\mathcal{F})$ by induction on $k$; the base case is trivial.
Let $\mathcal{F}=(a_{\sigma(1)}, a_{\sigma(2)}, \dots, a_{\sigma(k)})$ be a maximal increasing flag adapted to $A$ and $W$, by inductive step we assume that the flag $\mathcal{F}'= (a_{\sigma(1)}, a_{\sigma(2)}, \dots, a_{\sigma(k-1)})=(a_1, \dots, \widehat{a_{\sigma(k)}}, \dots, a_k)$ has labelling $\epsilon(\mathcal{F}')=(a_1, \dots, \widehat{a_{\sigma(k)}}, \dots, a_k)$.
The labelling $\epsilon(\mathcal{F})=(a_1, \dots, \widehat{a_{\sigma(k)}}, \dots, a_k, b)$ for some $b \in \{1,\dots,n\}$ is increasing but from the key observation we have $b \leq a_k$.
By definition of the labelling $b\geq a_k$, so $b=a_k$ and $\sigma(k)=k$.

Again by induction, we prove that the flag $\mathcal{F}=(a_1, a_2, \dots,a_k)$ has labelling $\epsilon(\mathcal{F})=(a_1, a_2, \dots,a_k)$ and so is increasing.
By inductive step $\epsilon((a_1, a_2, \dots,a_{k_1}))=(a_1, a_2, \dots,a_{k_1})$, so $\mathcal{F}$ has labelling $\epsilon(\mathcal{F})=(a_1, \dots, a_{k-1}, b)$ with $b\geq a_k$ by definition and with $b\leq a_k$ by the key observation.
We have proven that the flag $(a_1, a_2, \dots,a_k)$ is increasing.
\end{proof}

\begin{proof}[Proof of \Cref{thm:main_div}]
Let $g \colon H(T)[e_{W,A}] \to \DD$ be the map defined by $g(\chi^*(\omega))= \beta_\chi$ for all $\chi \in \Lambda$ and by $g(e_{W,A}) = \Xi_{W,A}$.
It is well defined since $\beta_{a\chi +b\eta}= a\beta_{\chi}+b \beta_\eta$ for all $a,b \in \Z$.
The ideal $I$ is contained in $\ker g$ by \Cref{lemma:multiplication,lemma:facile,lemma:circuit}, so $g\colon H(T)[e_{W,A}]/I \to \DD$ is well defined. 

We will show the injectivity of $g$ considering it as morphism of $H^\bigcdot(T)$-module.
Consider the monomial base of $\DD^{0,\bigcdot}$ provided in the second part of \Cref{lemma:monomial_base}.
Notice that in the expansion of $g(e_{W,A})=\Xi_{W,A}$, for $A$ no broken circuit set in $\mathcal{L}_{\leq W}$, appears only one monomial $\prod_{L \in \mathcal{F}} \sigma_L$ with $\mathcal{F}$ increasing chain (in $\mathcal{L}_{\leq W}$) 
by \Cref{lemma:increasing_flag}.
For each $W\in \mathcal{L}$ and $A$ no broken circuit set in $\mathcal{L}_{\leq W}$, we choose a set $B(A)$ such that $A \sqcup B(A)$ is a basis and a cone $C(A) \in \Delta$ contained in $\Ann \Lambda_A$ of maximal dimension.
Let us suppose that 
\[g \left( \sum_{W \in \mathcal{L}} \sum_{A \textnormal{ n.b.c.\ in } \mathcal{L}_{\leq W}} \alpha_{W,A} e_{W,A} \psi_{W,A} \right) =0\]
for some $\psi_{W,A} \in H^\bigcdot(W)$ and some $\alpha_{W,A} \in \Q$  with at least one $\alpha_{W,A}$ different from zero.
Let $(\overline{W}, \overline{A})$ such that $\lvert \overline{A} \rvert$ is maximal among all $(W,A)$ with $\alpha_{W,A} \neq 0$.
Let $\psi \in H(T)$ such that $\psi_{\overline{W},\overline{A}} \psi|_{\overline{W}}= \psi_{B(\overline{A})}$.
Let $\overline{\mathcal{F}}$ be the list of all elements in $\overline{A}$ ordered increasingly.
By \Cref{lemma:increasing_flag}, in
\begin{align*}
g \Big( \sum_{W \in \mathcal{L}} &\sum_{A \textnormal{ n.b.c.\ in } \mathcal{L}_{\leq W}} \alpha_{W,A} e_{W,A} \psi_{W,A} \psi \Big) = \sum_{W \in \mathcal{L}} \sum_{A \textnormal{ n.b.c.}} \alpha_{W,A} \Xi_{W,A} g(\psi_{W,A} \psi) \\
&= \sum_{W \in \mathcal{L}} \sum_{A \textnormal{ n.b.c.}} \sum_{\mathcal{F} \textnormal{ adap.\ } A,W} \alpha_{W,A} \frac{m(\mathcal{F})}{m(A)} \Big(\prod_{a \in A} x(\mathcal{F},a) \Big) g(\psi_{W,A} \psi) 
\end{align*}
the monomial $z=\prod_{L\in \overline{\mathcal{F}}}\sigma_L \prod_{j \in C(\overline{A})} b_j$ associated to the increasing flag $\overline{\mathcal{F}}$, can appear only in the addendum 
$\alpha_{\overline{W}, \overline{A}} \Xi_{\overline{W}, \overline{A}} g(\psi_{\overline{W}, \overline{A}} \psi).$
In particular $z$ appears only in the expansion of 
\begin{align*}
(\prod_{L \in \overline{F}} \sigma_L) g(\psi_{\overline{W}, \overline{A}} \psi) &= (\prod_{L \in \overline{F}} \sigma_L) g(\psi_{B(\overline{A})}) \\
&= \prod_{L \in \overline{F}} \sigma_L \prod_{b \in B(\overline{A})} \beta_b \\
&= m(B(\overline{A})) \prod_{L \in \overline{F}} \sigma_L  \sum_{ K \in \mathcal{C}_\Delta^{\lvert A\rvert}} \Vol (\pi(K)) \prod_{c \in K} b_c.
\end{align*}
The coefficient of $z$ in $(\prod_{L \in \overline{F}} \sigma_L) g(\psi_{\overline{W}, \overline{A}} \psi)$  must be zero, but it 
is (up to a sign) equal to $\alpha_{\overline{W}, \overline{A}} m(B(\overline{A})) \Vol (\pi(C(\overline{A})))$ c.f.\ \Cref{cor:beta_cones}.
The volume $\Vol (\pi(C(\overline{A})))$ is different from zero because $\Lambda_{\overline{A}} \otimes \Q \oplus \Lambda_{B(\overline{A})}\otimes \Q = \Lambda \otimes \Q$.
We have $\alpha_{\overline{W}, \overline{A}}=0$ contradicting the assumption, hence $g$ is injective.

Notice that the range of $g$ is contained in $\ker \dd$ by \Cref{lemma:in_ker} and in the subalgebra $\DD^{0,\bigcdot}$.
The map $g$ induces an injective map
\[g \colon \faktor{H^\bigcdot(T)[e_{W,A}]}{I} \to H^\bigcdot(\DD, \dd) \cong H^\bigcdot(M(\A);\Q)\]
since $\dd$ is of bidegree $(2,-1)$.
It is also surjective because $H^\bigcdot(T)[e_{W,A}]/I$ and $H^\bigcdot(M(\A);\Q)$ have the same dimension (see \Cref{lemma:nbc_basis}).
We have proven the Theorem.
\end{proof}

\begin{remark}
\Cref{thm:main_div} is a generalization of \cite[Theorem 5.2]{DCP05} and analogous to \cite[Theorem 6.13]{CDDMP19}.
Indeed, if $\A$ is totally unimodular and the circuit $C=\{0,1,\dots, n\}$ is oriented with $c_0=-$, $c_i=+$ for $i>0$, we obtain the Equation (20) of \cite{DCP05}.

We have chosen the generator associated with an hypertorus $S_a$ as $\Xi_{S_a,\{a\}} = \sigma_{S_a}+\beta_{\chi_a}^-$ that depends on the choice of one between $\chi_a$ and $-\chi_a$.
Another possible choice of generators were $\Xi_{S_a,\{a\}} = 2\sigma_{S_a}+\beta_{\chi_a}^- + \beta_{\chi_a}^+$, this would be lead to the same presentation of \cite[Theorem 6.13]{CDDMP19}.
\end{remark}

\begin{remark}
\Cref{thm:main_div} gives another proof of the rational formality of toric arrangements, previously proven in \cite{Dupont,CDDMP19}.
\end{remark}

\begin{conjecture}\label{conj-int}
Substituting in eq.~\eqref{eq:relazione_final} $\frac{m(A)}{m(X\setminus \{j\})} \psi_B$ with $\prod_{i=1}^{\lvert B\rvert} \psi_{\chi_i}$, where $(\overline{\chi_i})_i$ form a basis of $\Lambda_{C}/\Lambda_{A}$ with the same orientation of $(\overline{\chi}_b)_{b\in B}$, the cohomology ring with integer coefficients 
have a presentation analogous to the one in \Cref{thm:main_div}.
\end{conjecture}

This approach to the computation of cohomology ring for toric arrangements may be generalised to the non divisorial case.

\bibliography{Subspacesbib}{}
\bibliographystyle{amsalpha}
\bigskip \bigskip

\end{document}